\documentclass[12pt,a4paper]{amsart}
\usepackage{amsfonts,amscd,amsmath,amssymb,amsthm,mathrsfs}
\usepackage{array,latexsym,float,enumerate}
\usepackage{graphicx,epsfig,xcolor,ifpdf}
\usepackage[top=3cm, bottom=3cm, left=2.5cm, right=2.5cm,twoside=false]{geometry}
\usepackage[all]{xy}
\usepackage[normalem]{ulem} 
\numberwithin{equation}{section}

\theoremstyle{plain}
\newtheorem{theorem}{Theorem}[section]
\newtheorem{lemma}[theorem]{Lemma}
\newtheorem{proposition}[theorem]{Proposition}
\newtheorem{corollary}[theorem]{Corollary}
\theoremstyle{definition}
\newtheorem{definition}[theorem]{Definition}
\newtheorem{example}[theorem]{Example}
\newtheorem{noname}[theorem]{}
\newtheorem{subnoname}{}[theorem]
\newtheorem{remark}[theorem]{Remark}
\newtheorem{construction}[theorem]{Construction}
\newtheorem{notation}[theorem]{Notation}
\theoremstyle{remark}
\newtheorem*{smallremark}{Remark}

\newtheorem{case}{Case} \makeatletter \@addtoreset{case}{theorem}\makeatother
\newtheorem{claim}{Claim} 

\newcommand{\bthm}{\begin{theorem}}
\newcommand{\bprop}{\begin{proposition}}
\newcommand{\blem}{\begin{lemma}}
\newcommand{\bcor}{\begin{corollary}}
\newcommand{\brem}{\begin{remark}}
\newcommand{\bdfn}{\begin{definition}}
\newcommand{\bitem}{\begin{itemize}}
\newcommand{\benum}{\begin{enumerate}}
\newcommand{\bex}{\begin{example}}
\newcommand{\bno}{\begin{noname}}
\newcommand{\bsno}{\begin{subnoname}}
\newcommand{\bsrem}{\begin{smallremark}}
\newcommand{\bnot}{\begin{notation}}
\newcommand{\bcon}{\begin{construction}}
\newcommand{\bca}{\begin{case}}
\newcommand{\bcl}{\begin{claim}}
\newcommand{\beq}{\begin{equation}}
\newcommand{\bpf}{\begin{proof}}

\newcommand{\epf}{\end{proof}}
\newcommand{\eeq}{\end{equation}}
\newcommand{\ecl}{\end{claim}}
\newcommand{\eca}{\end{case}}
\newcommand{\econ}{\end{construction}}
\newcommand{\enot}{\end{notation}}
\newcommand{\esrem}{\end{smallremark}}
\newcommand{\eno}{\end{noname}}
\newcommand{\esno}{\end{subnoname}}
\newcommand{\eex}{\end{example}}
\newcommand{\eitem}{\end{itemize}}
\newcommand{\eenum}{\end{enumerate}}
\newcommand{\ethm}{\end{theorem}}
\newcommand{\eprop}{\end{proposition}}
\newcommand{\elem}{\end{lemma}}
\newcommand{\ecor}{\end{corollary}}
\newcommand{\erem}{\end{remark}}
\newcommand{\edfn}{\end{definition}}
\newcommand{\ble}{\begin{lemma}}
\newcommand{\ele}{\end{lemma}}

\newcommand{\ov}{\overline}

\newcommand{\wt}{\widetilde}
\newcommand{\cal}[1]{\mathcal{#1}}

\newcommand{\ds}{\displaystyle}

\def\8{\infty}
\def\.{\cdot}
\def\PP{\mathbb{P}}
\def\F{\mathbb{F}}
\def\C{\mathbb{C}}
\def\Z{\mathbb{Z}}

\def\Q{\mathbb{Q}}

\def\ovk{\overline\kappa}

\newcommand{\To}{\rightarrow}

\def\mono{\hookrightarrow}

\def\:{\colon}

\renewcommand{\iff}{\Leftrightarrow}

\newcommand{\noin}{\noindent}
\def\Bk{\operatorname{Bk}}

\def\Sing{\operatorname{Sing}}

\def\Supp{\operatorname{Supp}}

\def\Bk{\operatorname{Bk}}

\def\Sing{\operatorname{Sing}}

\def\Supp{\operatorname{Supp}}

\def\Spec{\operatorname{Spec}}
\def\NS{\operatorname{NS}}

\def\id{\operatorname{id}}

\newcommand{\ol}{\overline}

\newcommand{\ti}{\tilde}

\def\ks{K_{\ov S}}

\ifpdf \usepackage[linkbordercolor={0 0 1}]{hyperref} \else \usepackage[hypertex,linkbordercolor={0 0 1}]{hyperref} \fi

\begin{document}

\title[The geometry of sporadic $\C^*$-embeddings into $\C^2$]{The geometry of sporadic $\C^*$-embeddings into $\C^2$}

\author[Mariusz Koras]{Mariusz Koras}
\address{Mariusz Koras: Institute of Mathematics, University of Warsaw, ul. Banacha 2, 02-097 Warsaw}\email{koras@mimuw.edu.pl}
\author[Karol Palka]{Karol Palka}
\address{Karol Palka: Institute of Mathematics, Polish Academy of Sciences, ul. \'{S}niadeckich 8, 00-956 Warsaw, Poland}
\email{palka@impan.pl}
\author[P.Russell]{Peter Russell}
\address{Peter Russell: McGill University, Montreal, Canada}\email{russell@math.mcgill.ca}

\thanks{The first and the second author were supported by the National Science Center, Poland, Grant No. 2013/11/B/ST1/02977. The second author was partially supported by the Foundation for Polish Science under the Homing Plus programme, cofinanced from the European Union, RDF. The third author was supported by NSERC, Canada.}


\begin{abstract} A closed algebraic embedding of $\C^*=\C^1\setminus\{0\}$ into $\C^2$ is \emph{sporadic} if for every curve $A\subseteq \C^2$ isomorphic to an affine line the intersection with $\C^*$ is at least $2$. Non-sporadic embeddings have been classified. There are very few known sporadic embeddings. We establish geometric and algebraic tools to classify them based on the analysis of the minimal log resolution $(X,D)\to (\PP^2,U)$, where $U$ is the closure of $\C^*$ on $\PP^2$. We show in particular that one can choose coordinates on $\C^2$ in which the type at infinity of the $\C^*$ and the self-intersection of its proper transform on $X$ are sharply limited.
\end{abstract}

\maketitle

\section{Main result and discussion}

We continue the analysis of closed algebraic embeddings of $\C^*=\C^1\setminus \{0\}$ into the complex affine plane $\C^2$ initiated in \cite{CKR-Cstar_good_asymptote}, where embeddings admitting a good asymptote have been classified.

\begin{definition}\label{asymptote} Let $U\subset \C^2$ be a closed curve isomorphic to $\C^*$. A curve $A\subset \C^2$ isomorphic to the affine line $\C^1$ is called a \textit{good asymptote} of $U$ if and only if $A\cdot U\leq 1$. If $U\subset \C^2$ does not admit a good asymptote we call the embedding \emph{sporadic}.\end{definition}

Note that the intersection is taken in $\C^2$, so the definition is independent of the choice of coordinates, i.e.\ of a choice of generators of the algebra of regular functions on $\C^2$. Surprisingly, although the defining condition of sporadic embeddings seems to be weak, up to now we know only very few of them: one discreet family with no deformations and one more embedding, see \cite[Main Theorem (s),(t)]{BoZo-annuli} (note that the list in loc.\ cit.\ is produced assuming strong 'regularity condition'). The goal of this article is to establish geometric and algebraic machinery which allows to prove strong restrictions on sporadic $\C^*$-embedding in terms of the resolution of singularities of their closures on $\PP^2$. With these tools in hand we are going to obtain the full classification in a forthcoming paper.

We introduce the following numbers characterizing the embedding. Let $(\lambda,P)$ be an analytically irreducible germ of a planar curve and let $L$ be a curve smooth at $P$ which does not cross $\lambda$ normally (i.e.\ which is tangent to $\lambda$ at $P$ in case $\lambda$ is smooth). The \emph{jumping number} $j(\lambda, L)$ of $\lambda$ \emph{with respect to L} is the maximal number of blowups on the proper transform of $\lambda$ after which $\lambda$ meets the total transform of $L$ not in a node. In particular, $j(\lambda, L)=0$ if and only if $L$ is tangent to $\lambda$. If $\lambda$ is singular and $(\eta,P)$ is a smooth germ maximally tangent to $\lambda$, i.e., such that $\lambda\cdot \eta$ is maximal for intersections of $\lambda$ with smooth germs, then $j(\lambda, L)$ is the integral part of $\lambda\cdot \eta$/$\lambda\cdot L$. We write $\C^2=\PP^2\setminus L_\8$, where $L_\8$ is a line (degree $1$ curve) in $\PP^2$, called the \emph{line at infinity}.

\begin{definition}\label{def:type} Let $U\subset \C^2=\PP^2\setminus L_\8$ be a closed curve and let $\lambda_1, \lambda_2,...,\lambda_k$ be the germs at
infinity of the closure $\bar U$ of $U$ in $\PP^2$ ordered so that the jumping numbers $j_i=j(\lambda_i,L_\8)$ do not decrease with $i$.  We then call
$(j_1,j_2,...,j_k)$ the \textit{type at infinity of} $U$.
\end{definition}

Note that the identification $\C^2=\PP^2\setminus L_\8$ and the type of a curve at infinity depend on a choice of coordinates. Another number associated to an embedding $U\subset \C^2=\PP^2\setminus L_\8$ is the self-intersection of the proper transform of the closure $\bar U$ of $U$ under the minimal log resolution of singularities of $(\PP^2,\bar U +L_\8)$, i.e., after the minimal number of blowups so that the total transform of $\bar U +L_\8$ is a simple normal crossing divisor. For $U\cong \C^*$ our main result proves the existence of special coordinates with respect to which these numbers are sharply limited.
\begin{theorem}\label{thm:main_result} For every sporadic $\C^*$-embedding $U\subset \C^2$ one can choose coordinates on $\C^2$ so that:
\begin{enumerate}
\item the branches at infinity of $U$ are disjoint and the type at infinity of $U$ is $(1,\ti j)$ for some $\ti j\in\{2,3,4,5,6\}$ and
\item if $\bar U$ denotes the closure of $U$ on $\PP^2=\C^2\cup L_\8$ then the proper transform of $\bar U$ under the minimal log resolution of $(\PP^2,\bar U+L_\8)$ has self-intersection between $-2$ and $-5$.
\end{enumerate}
\end{theorem}

\bsrem In elementary planar geometry, a choice of coordinates on $\C^2$ having been made, an \textit{asymptote} of a closed curve $U$ is a straight line tangent to a branch of $U$ at infinity. If $A$ is a good asymptote in the sense of \ref{asymptote} then, since $A$ is an affine line, we can by the Abhyankar-Moh-Suzuki theorem choose coordinates so that $A$ is a straight line. Then, assuming $U$ is irreducible and of degree at least three, $A$ is tangent to $U$ at infinity. It follows that a good asymptote of $U\cong \C^*$ is an asymptote in suitable coordinates. Let $U\subseteq \C^2$ be a sporadic $\C^*$-embedding. Interestingly, although $U$ does not have a good asymptote, by \ref{thm:main_result}(1) both jumping numbers are positive, so the branches at infinity of $U$ are not tangent to the line at infinity. Therefore, $U$ has an asymptote at each of its two points at infinity, each of them meets $U$ at lest twice on $\C^2$, see \ref{asymptotes}.
\esrem

We now discuss our approach. The first step to understand the geometry of sporadic $\C^*$-embeddings was made in \cite{KR-Some_properties_of_Cstar}, where it has been proved that one can choose coordinates on $\C^2$ so that the closure of $U$ in $\PP^2$ meets the line at infinity in exactly two points, i.e., so that the two branches at infinity of $\C^*$ are separated.\footnote{We recall that in contrast there are $\C^*$-embeddings admitting a good asymptote for which the branches at infinity meet for every choice of coordinates, \cite[6.8.1]{CKR-Cstar_good_asymptote}.} Remembering that the known sporadic embeddings are given by very special equations (or parametrizations) which we need to somehow recover, from the very beginning we need to have a precise control over the singularities and their behaviour under subsequent steps of the log resolution $\Phi\: (\ov S',D'+E')\to (\PP^2,L_\8+\bar U)$, where $E'$ and $D'$ are respectively the proper transform of $\bar U$ and the reduced total transform of $L_\8$. This is achieved using Hamburger Noether pairs of the resolution (see sec. \ref{ssec:HN}) and two fundamental equations \eqref{eq:1} and \eqref{eq:2} relating them with properties of $E'$. Since the proper transform of $L_\8$ on $\ov S'$ may be a non-branching $(-1)$-curve, we need to consider an snc-minimalization $\Psi\:(\ov S',D'+E')\to (\ov S,D)$. The basic characteristic numbers are $\varepsilon=2-(K_{\ov S}+D)^2$ and $\gamma=-E^2$, where $E$ is the proper transform of $U$ on $\ov S$. They are bounded by $\gamma\geq 1$ and $\varepsilon\geq 0$ (\ref{lem:eps_and_gamma}). The nonexistence of a good asymptote turns out to have strong consequences: we obtain $2\varepsilon+\gamma\leq 9$ (\ref{prop:basic_inequality}) and we show that the pair $(\ov S,D)$ is almost minimal and of log general type (\ref{lem:eps_and_gamma}(iv), \ref{lem:S-E_has_k=2}). Further bounds are obtained in section \ref{sec:surgeries} by studying the geometry of the pair $(\ov S,Q=D+E-C-\ti C)$, where $C$ and $\ti C$ are the last curves produced by the log resolution over each branch at infinity. Importantly, the surface $Y=\ov S\setminus Q$ has negative Euler characteristic and the pair $(\ov S,Q)$ turns out to be almost minimal. These facts constitute a basis of the improved bound $\gamma\leq 5$ (see \ref{prop:gamma<=5}). In many proofs we heavily rely on the logarithmic version of the Bogomolov-Miyaoka-Yau inequality \ref{lem:BMY} and on structure theorems for quasi-projective surfaces of non-general type. The nonexistence of a good asymptote is used again via \ref{lem:c1B-p1>=2} in sections $5$ and $6$, where we bound the possible types of $U$ at infinity.

\tableofcontents

\section{Preliminaries}

\subsection{Surfaces, divisors and minimal models}\label{ssec:generalities} We recall some notions and results from the theory of open algebraic surfaces.  We refer the reader to \cite[\S2]{Miyan-OpenSurf} and \cite[\S3]{Fujita-noncomplete_surfaces} for details.

Let $X$ be a normal projective surface and $T=\sum_{i=1}^n m_iT_i$ a divisor contained in the smooth part of $X$ with $T_1,\ldots, T_n$ distinct, irreducible curves. We say that the pair $(X,T)$ is \emph{smooth} if $X$ is smooth and $T$ is a reduced simple normal crossing (snc) divisor. A $(b)$-curve on $X$ is a curve $L\cong \PP^1$ with $L^2=b$. An snc divisor $T$ is snc-minimal if a contraction of any $(-1)$-curve contained in $T$ leads to a divisor which is not snc. We call $Q(T)=(T_i\cdot T_j)_{1\leq i,j\leq n }$ the \emph{intersection matrix} of $T$ and we define the discriminant of $T$ by $d(T)=\det(-Q(T))$. We put $d(T)=1$ if $T=0$. If $T=T_1+\ldots+T_n$ is an ordered chain of rational curves (by definition the components of a chain are smooth) we write $T=[-T_1^2,\ldots,-T_n^2]$. We have, by elementary linear algebra, \begin{equation}\label{eq:d(T)_recurrence} d(T)=(-T_1^2)d(T_2+\ldots+T_n)-d(T_3+\ldots+T_n). \end{equation} The formula implies in particular that for \emph{admissible} chains, i.e.\ the ones with $T_i^2\leq -2$ for all $i$, one has $d(T-T_1)<d(T)$. A chain of $(-2)$-curves of length $m$ is denoted by $[(2)_m]$.

The linear equivalence of integral divisors and the numerical equivalence of divisors are denoted by $\sim$ and $\equiv$ respectively. If $\sigma\:X'\to X$ is a blowup then we say that it is \emph{inner} (\emph{outer}) for $T$ if the center belongs to exactly two (resp. exactly one) component of $T$.

\label{srem:weighted graph}Assume $T$ is a reduced snc divisor with connected support and a negative definite intersection matrix. We denote the local fundamental group of the analytic singularity arising from the contraction of $T$ by $\Gamma(T)$. If the singularity is of quotient type, i.e.\ analytically isomorphic to $0\in \C^2//G$ for some finite group $G<GL(2,\C)$, then we say that $T$ {\it is of quotient type}. In the latter case the singularity is algebraic (in fact rational). We allow the possibility that $G\cong \{1\}$, in which case the corresponding point is smooth. If $Q(T)$ is not negative definite we put formally $|\Gamma(T)|=\infty$. It is known (see \cite{Mumford}) that if $T$ has rational components then $\Gamma(T)$ depends only on the weighted dual graph of $T$. Also, $d(T)$ is the order of the first local integral homology group, which is the abelianization of $\Gamma(T)$. Moreover, $\Gamma(T)$ is finite if and only if the corresponding singularity is of quotient type.

Assume $T$ is of quotient type and snc-minimal. Then $T$ is a rational tree and contains at most one branching component (see \cite{Brieskorn}). If it does contain one (i.e.\ $T$ is a fork) then the corresponding singularity is non-cyclic, $T$ contains a unique branching component, which has self-intersection $b\leq -2$, and three rational admissible chains attached to it with discriminants $(d_1,d_2,d_3)=(2,3,3), (2,3,4), (2,3,5)$ or $(2,2,n)$ for some $n\geq 2$. The sequence $(-b;d_1,d_2,d_3)$ (and also the triple $(d_1,d_2,d_3)$), is called the \emph{type} of the fork $T$. If $T$ is a chain then either $T=[1]$ (if $\Gamma(T)=\{\id\}$) or $T$ is admissible and the corresponding singularity is cyclic.

We denote the Neron-Severi group of $X$ by $\NS(X)$ and its rank by $\rho(X)$. The number of irreducible components of a divisor $D$ is denoted by $\#D$. The following lemma is due to Fujita \cite[4.16]{Fujita-noncomplete_surfaces}.

\begin{lemma}\label{lem:rulings}
Assume $(\ov X,D)$ is a smooth pair and $\pi:\ov X\to B$ a $\PP^1$-fibration onto a smooth curve. Put $X=\ov X\setminus D$. Let $h$ be the number of horizontal components of $D$ and $\nu$ the number of fibers contained in $D$. Put $\Sigma_X=\sum_{F\not\subseteq D}(\sigma(F)-1)$, where $\sigma(F)$ is the number components of a fiber $F$ not contained in $D$ and where the sum is taken over all fibers not contained in $D$. Then the following relation holds $$\Sigma_X=h+\nu-2+\rho(X)-\#D.$$ \end{lemma}



As a consequence of the Hodge index theorem we have the following lemma.

\begin{lemma}\label{lem:Hodge lemma}Let $C_1,\ldots, C_r$ be distinct irreducible curves on a smooth projective surface $X$. If the matrix $(C_i\cdot C_j)_{i,j\leq r}$ is negative definite then $r< \rho(X)$.
\end{lemma}

Recall that in dimension $2$ running the log Minimal Model Program for a smooth pair $(X,D)$ results with a morphism onto a minimal model $\alpha_m\:(X,D)\to (X_m,D_m)$ for which the log surface $(X_m,D_m)$ is log terminal, hence $X_m$ has only quotient singularities. If $\tau\:(X_a,D_a)\to (X_m,D_m)$ is the minimal log resolution then there is a lift $\alpha_a\:(X,D)\to (X_a,D_a)$ of $\alpha_m$. The (smooth) pair $(X_a,D_a)$ is called an \emph{almost minimal model} of $(X,D)$ and the morphism $\alpha_a$ is well described (see \cite[2.3.11]{Miyan-OpenSurf}). If $D$ is snc-minimal and $X\setminus D$ is affine then $X_a\setminus D_a$ is an open subset of $X\setminus D$ with the complement being a disjoint sum of a finite number of curves isomorphic to $\C^1$ and $\tau$ contracts exactly the maximal admissible rational twigs of $D_a$. If $\kappa(K_X+D)\geq 0$ then the divisor $\tau^*(K_{X_m}+D_m)$ is the positive part of the Zariski-Fujita decomposition of $K_{X_a}+D_a$. The negative part can be described very explicitly in terms of a \emph{bark} of a divisor as follows (for details see \cite[\S3]{Fujita-noncomplete_surfaces} and \cite[\S2, section 3]{Miyan-OpenSurf}).

If $R=R_1+\ldots+R_s$ is a linearly ordered admissible chain we put $$e(R)=d(R_2+\ldots +R_s)/d(R).$$ Let $D$ be a reduced snc-minimal divisor with connected support. First assume $D$ is not of quotient type. Let $T_1,\dots, T_s$ be the maximal admissible twigs of $D$. By convention the components of $T_i$ are ordered linearly so that the tip of $T_i$ is its first component (the tip is the component meeting only one other component of $D$). In this case we put $e(D)=\sum_{i=1}^s e(T_i)$ and we define the \emph{bark of $D$} as the unique $\Q$-divisor supported on $\bigcup_i \Supp T_i$ and satisfying $$(K+D-\Bk D)\cdot Z=0$$ for every component $Z$ of $T_1+\ldots+T_s$. Equivalently, $Z\cdot \Bk D$ is $-1$ if $Z$ is a tip of $D$ and $0$ otherwise. One checks that $$\Bk^2 D=-e(D).$$ Now assume $D$ is of quotient type. In this case we define \emph{bark of $D$} as the unique $\Q$-divisor supported on $D$ satisfying $(K+D-\Bk D)\cdot Z=0$ for every component $Z$ of $D$. If $D=D_1+D_2+\ldots+ D_n$ is a chain then $$-\Bk^2(D)=e(D_1+D_2+\ldots+ D_n)+e(D_n+D_{n-1}+\ldots+ D_1).$$ If $D$ is a fork then the formula is a bit more complicated but we will not need it. For a general (reduced snc-minimal) $D$ one simply defines $\Bk D$ by summing barks of connected components of $D$. One shows that $\Bk D$ is an effective $\Q$-divisor with proper fractional coefficients and support equal to the sum of connected components of quotient type and the sum of maximal admissible twigs of the remaining connected components. What is most important for us is that (see (\cite[\S2, section 3]{Miyan-OpenSurf}) if $(X,D)$ is an almost minimal smooth pair with $\kappa(K_X+D)\geq 0$ then $$(K_X+D)^-=\Bk D.$$

The following version of the logarithmic Bogomolov-Miyaoka-Yau inequality follows from \cite{Langer} (for how it follows see \cite[5.2]{Palka-exceptional}). The original, weaker version was proved by Kobayashi-Nakamura-Sakai. The Euler characteristic of a topological space $Z$ will be denoted by $\chi(Z)$.

\begin{lemma}[The log BMY inequality]\label{lem:BMY}
Let $(X,D)$ be a smooth pair and let $D_1,\dots, D_k$ be the connected components of $D$ which are of quotient type. If $(X,D)$ is almost minimal and $\kappa(X\setminus D)\geq 0$ then \[\frac{1}{3}((K_{X}+D)^+)^2\leq \chi(X\setminus D)+\sum_{i=1}^k\frac{1}{|\Gamma(D_i)|},\] \end{lemma}

\subsection{The log resolution}\label{ssec:resol}

\begin{notation}\label{not:setup} Let $S=\C^2=\Spec \C[x,y]$. Let $U$ be a $\C^*$ embedded as a closed subset of $S$. Having coordinates $(x,y)$ on $\C^2$ we have an identification $\C^2=\PP^2\setminus L_\8$, where $L_\8$ is a line on $\PP^2$. We denote by $\bar U$ the closure of $U$ in $\PP^2$. Let $\lambda$ and $\tilde \lambda$ denote the branches of $\bar U$ at infinity i.e.\ the germs of $\bar U$ at $\bar U\cap L_\infty$. Note that an automorphism $\alpha$ of $\C^2$ gives rise to new coordinates $(\alpha^*x,\alpha^*y)$ on $\C^2$. It is proved in \cite{KR-Some_properties_of_Cstar} that if $U$ does not admit a good asymptote then there is a choice of coordinates on $\mathbb{C}^2$, such that $\lambda$ and $\tilde\lambda$ are disjoint. From now on throughout  the paper we assume that it is the case.

Let $$\Phi\colon (\ov{S}',D'+E')\rightarrow (\mathbb{P}^2,L_\8+\bar U),$$ where $D'$ is the reduced total transform of $L_{\infty}$ and $E'$ is the proper transform of $\bar U$, be the minimal log resolution of singularities. By definition $\Phi^{-1}$ is the minimal sequence of blow-ups such that $D'+E'$ is an snc divisor. Let $L_\8'\subseteq D'$ be the proper transform of $L_\8$ in $\ov S'$ (see Fig. \ref{fig:Sbarprim}).
\begin{figure}[h] \includegraphics[scale=0.53]{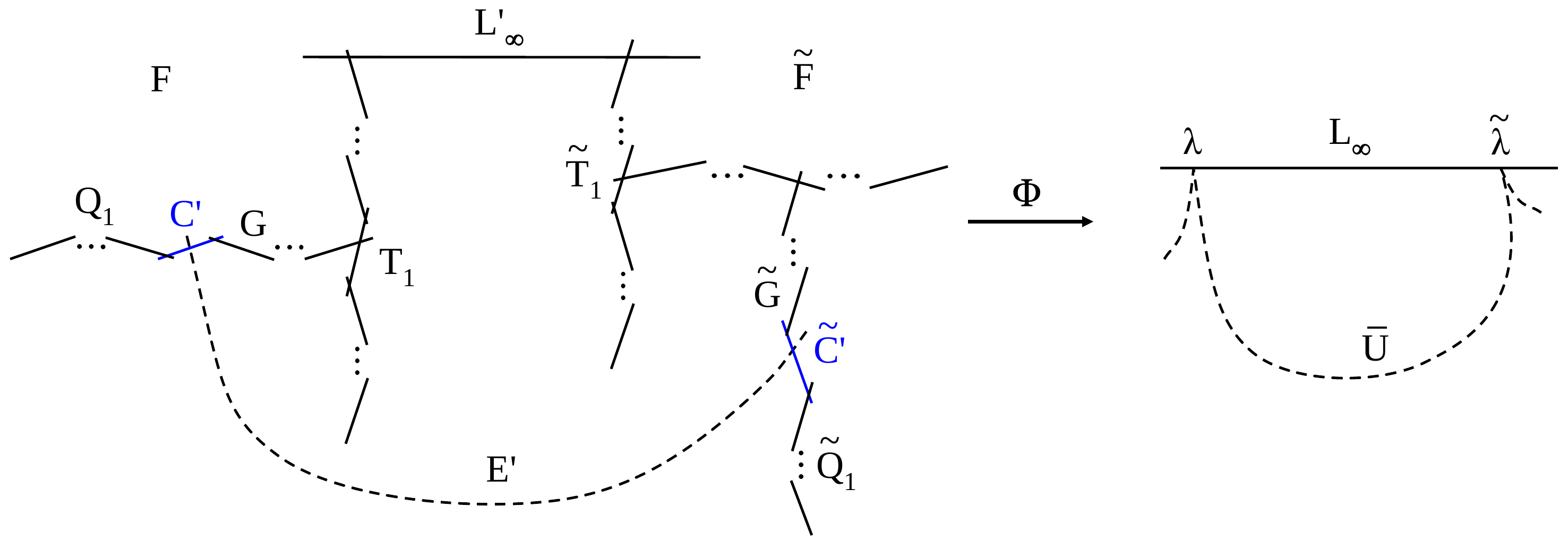}\caption{The log resolution $\Phi\: (\ov S',D'+E')\to (\PP^2,L_\8+\bar U)$.}  \label{fig:Sbarprim}\end{figure}

\noin It may happen that $L_\8'$ is a $(-1)$-curve. Let $$\Psi\: (\ov{S}',D'+E')\to (\ov{S},D+E),$$ where $D=\Psi_*D'$ and $E=\Psi_*E'$, be the {\it snc minimalization of the divisor $D'$ with respect to $E'$}, i.e., $\Psi$ is the identity if $L_{\infty}'$ is not a $(-1)$-curve and otherwise it is the composition of successive contractions of $L_{\infty}'$ and then possibly of other $(-1)$-curves in the successive images of $D'$, such that $D+E$ is an snc divisor and each $(-1)$-curve of $D$ is a branching component of $D+E$. Now the only case when $D+E$ is not snc-minimal is when $E^2=-1$. Of course, $\ov{S}\setminus D=\ov S'\setminus D'=S=\C^2$ and $E'\cdot D'=E\cdot D=2$. Put $\gamma=-E^2$, $\gamma'=-(E')^2$ and define $\varepsilon$ by the equality $$(\ks+D+E)^2=2-\varepsilon.$$
\end{notation}

The assumption that a good asymptote for $U\subseteq S$ does not exist can be restated as follows.

\begin{lemma}\label{lem:no asymptote}
There is no curve $L\subset \ov S$ such that $L\cap S\cong \C^1$ and $L\cdot E\leq 1$. \end{lemma}

\subsection{The Hamburger-Noether pairs}\label{ssec:HN}

We write $D'=L'_\infty+F+\tilde F$ where $F=\Phi^{-1}(\lambda\cap L_\infty)_{red}$ and $\tilde F=\Phi^{-1}(\tilde\lambda\cap L_\infty)_{red}$. Let $C',\ti C'$ be the components of $D'$ meeting $E'$ contained in $L_\8'\cup F$ and $L_\8'\cup \tilde F$ respectively. Note that it may happen that $F=0$ (or $\ti F=0$). This means that the branch $\lambda$ crosses $L_\8$ normally, i.e.\ it is smooth and transversal to $L_\8$. We call such a branch \emph{simple}. The resolution process $\Phi$, see \ref{not:setup}, can be described in terms of Hamburger-Noether (HN-) pairs. For details we refer to \cite[Appendix]{KR-C*_actions_on_C3} or \cite{Russell_HN_pairs}. By $(T\cdot Z)_p$ we denote the local intersection index of two curves $T,Z$ at a point $p$.

We describe the process for $\lambda$, for $\ti \lambda$ it is analogous. As an input data we have a locally analytically irreducible branch $(\lambda_1,q_1)=(\lambda,\lambda\cap \bar U)$ and a curve $L_\8$ smooth at $q_1$.  Let $x_1$ be a coordinate defining $L_{x_1}=L_\8$ at $q_1$. Put $$c_1=(\lambda_1\cdot L_{x_1})_{q_1}.$$ If $c_1=1$ (equivalently, if $\lambda$ and $L_\8$ cross normally at $q_1$) we put $p_1=0$, $h=0$ and we do nothing. Assume $c_1>1$.  We then pick $y_1$ so that $(x_1, y_1)$ is a system of parameters at $q_1$ and $$p_1 = (\lambda_1\cdot L_{y_1})_{q_1}$$ is the multiplicity of $\lambda_1 \ \ \text {at} \ \ q_1.$ This forces  $c_1\geq p_1$. We blow up successively over $q_1$ until the proper transform $\lambda_2 $ of $\lambda_1$ meets the inverse image of the divisor $L_{x_1}+L_{y_1}$ not in a node. The exceptional curves form a chain  called {\it {the chain produced by the pair}} $\binom{c_1}{p_1}$. Let $C_1$ be the last exceptional curve. We then say that $C_1$ is the {\it {exceptional curve produced by the pair}} $\binom{c_1}{p_1}$. Since $\lambda_1$ is locally analytically irreducible, $\lambda_2$ meets $C_1$ in a unique point $q_2$ and does not meet any other component of the inverse image of $L_{x_1}+L_{y_1}$. We choose for $x_2$ a local coordinate for $C_1$ at $q_2$ and continue the process, noting that $c_2= \lambda_2 \cdot C_1 =\gcd(c_1, p_1)$. We continue this process until the proper transform of $\lambda_1$ meets the last exceptional curve normally. This describes the log resolution of $(\PP^2,\lambda_1 \cup L_{x_1})$ at $q_1$ and $$\binom{c_1}{p_1},\binom{c_2}{p_2},\dots ,\binom{c_h}{p_h}$$ are called the \emph{HN-pairs of $(\lambda_1,L_{x_1})$}. We note that $$c_{i+1}=\gcd(c_i,p_i)\text{\ for\ } i=1,\ldots,h-1 \text{\ and \ } \gcd(c_h,p_h)=1.$$

After a slight change of numbering we may write the sequences of HN-pairs of $\lambda$ and $\tilde\lambda$ as
\begin{equation}\binom{c_1}{c_1}_j,\binom{c_1}{p_1},\dots,\binom{c_h}{p_h}\text{\ \ and\ \ } \binom{\ti c_1}{\ti c_1}_{\ti j},\binom{\tilde c_1}{\tilde p_1},\dots, \binom{\tilde c_{\tilde h}}{\tilde p_{\tilde h}},\label{eq:def_of_pairs}\end{equation} where $p_1<c_1$ and $\ti p_1<\ti c_1$ and by $\binom{c}{c}_j$ we mean a sequence of pairs $\binom{c}{c},\ldots,\binom{c}{c}$ of length $j$. Moreover, interchanging the names $\lambda$ and $\ti \lambda$ if necessary, we may, and shall, assume that $\ti j\geq j$. As a consequence of the minimality of the resolution we have  $p_h<c_h$ and $\ti p_{\ti h}<\ti c_{\ti h}$.

\bsrem  The HN-pairs for $U\subseteq \C^2$ depend in general on the choice of coordinates on $\C^2$. We remark that to each HN-pair there is tacitly associated a complex number that determines the location of the branch on the last exceptional curve produced by the blowups prescribed by the pair; we will not make use of it.\esrem

Let $\mu_1,\mu_2,\ldots$  (resp. $\tilde \mu_1,\tilde \mu_2,\ldots$) be the sequence of multiplicities of all  singular
points of $\lambda $ infinitely near $\lambda\cap L_\infty$ (resp.  of $\tilde\lambda$ infinitely near  $\tilde\lambda\cap L_\infty$). Then (\cite[Appendix]{KR-C*_actions_on_C3})
\begin{align*} \sum_{i \geq  1}\mu_i&=(j+1) c_1+p_1+p_2+\cdots +p_h-1.\\
\sum_{i \geq  1}\mu_i^2&=jc_1^2+c_1p_1+c_2p_2+\cdots +c_hp_h.\\
\sum_{i \geq  1}\tilde \mu_i&=(\ti j+1)\tilde c_1+\tilde p_1+\tilde p_2+\cdots +\tilde p_{\tilde h}-1.\\
\sum_{i \geq  1}\tilde\mu_i^2&=\ti j \ti c_1^2+\tilde c_1\tilde p_1+\tilde c_2\tilde p_2+\cdots + \tilde c_{\tilde h}\tilde p_{\tilde h}.\end{align*}

\begin{lemma}Let $d$ be the degree of $\bar U\subseteq \PP^2$. The following equations hold:
\begin{align} d&=c_1+\ti c_1,\\
2d+\gamma'&=jc_1+\sum_{i=1}^h p_i+\ti j \ti c_1+\sum_{i=1}^{\ti h} \ti p_i,\label{eq:1}\\
d^2+\gamma'&=jc_1^2+\sum_{i=1}^h c_ip_i+\ti j \ti c_1^2+\sum_{i=1}^{\ti h} \ti c_i\ti p_i,\label{eq:2}\\
 h+j+\ti h+\ti j&=\gamma'+\varepsilon+2.\label{eq:h+hB}\end{align}
\end{lemma}

\begin{proof} By the definition of $c_1, \ti c_1$ we have  $d=L_\8\cdot \bar U=L_\8\cdot(\lambda+\ti \lambda)=c_1+\ti c_1$. Tracking the self-intersection and intersections of the proper transforms of $\bar U$ with canonical divisors under blowups constituting the resolution $\Phi$ we get $K_{\ov S'}\cdot E'-K_{\PP^2}\cdot \bar U=\ds\sum_{i \geq  1}\mu_i+\ds\sum_{i \geq  1}\ti \mu_i$ and $\bar U^2-(E')^2=\ds\sum_{i \geq  1}\mu_i^2+\ds\sum_{i \geq  1}\ti \mu_i^2$. We compute $K_{\ov S'}\cdot E'-K_{\PP^2}\cdot \bar U=\gamma'-2+3d$ and $\bar U^2-(E')^2=d^2+\gamma'$, where $\gamma'=-(E')^2$. Using the first equation and the equations preceding the lemma we obtain the second and third equation.

The contractions in $\Psi$ are inner for $D'+E'$, so $\ks\cdot(\ks+D+E)=K_{\ov S'}\cdot (K_{\ov S'}+D'+E')=K_{\ov S'}\cdot E'+K_{\PP^2}\cdot(K_{\PP^2}+L_\8)-(h+j+\ti h+\ti j)=\gamma'+4-(h+j+\ti h+\ti j)$. On the other hand, since the arithmetic genus of $D+E$ vanishes, we have $\ks\cdot(\ks+D+E)=(\ks +D+E)^2=2-\varepsilon$. This gives the fourth equation.
\end{proof}

Recall that we  assume that the branches at infinity of the closure of $U$ are separated. Note that if we blow according to a HN-pair $\binom{c}{p}$ then after making the first blowup, the branch either 'stays' on a given irreducible component of the boundary in case $c=p$, i.e.\ its proper transform meets the proper transform of the component, or it 'jumps', i.e.\ it separates from it in case $c>p$. Hence the following holds.

\blem  The pair $(j,\ti j)$ defined above is the type of $U$ at infinity in the sense of \ref{def:type}.
\elem

\brem\label{rem:capacity_using_pairs}  It is an elementary exercise to show that the maximal twig of $D+E'$ created by a pair $\binom{c_i}{p_i}$ has discriminant $c_i/\gcd(c_i,p_i)$ and its contribution to $e(D+E')$ is $(c_i-p_i)/c_i$.
\erem

\section{Basic results}

\subsection{Basic inequalities} We use the notation from the previous section.

\begin{lemma}\label{lem:eps_and_gamma} We have:
	 \begin{enumerate}[(i)]

    \item $\ks\cdot (\ks+D+E)=2-\varepsilon,$

	\item $\varepsilon\geq 0$,

    \item $\gamma\geq 1$.

    \item If $\gamma\neq 1$ then $(\ov S,D+E)$ is almost minimal and $(\ks+D+E)^-=\Bk {(D+E)}$.

    \item $-\Bk^2(D+E)=e(D+E)\leq 1+\varepsilon$.
\end{enumerate}
\end{lemma}

\begin{proof} (i) $\ks\cdot (\ks+D+E)=(\ks+D+E)^2-2(p_a(D+E)-1)=(\ks+D+E)^2$.

(ii) Note that the support of $D+E$ contains a loop, so $|\ks+D+E|\neq \emptyset$ and hence $\kappa(\ks+D+E)\geq 0$ (see \cite[1.2.9]{Miyan-OpenSurf}). Let $\cal N$ be the negative part in the Zariski decomposition of the divisor $\ks+D+E$. By \cite{Miyaoka} (or by \cite[Corollary 5.2]{Langer}), $(\ks +D+E)^2\leq 3\chi(S\setminus E)+\frac{1}{4}\cal N^2=3+\frac{1}{4}\cal N^2$, so $\varepsilon\geq -1-\frac{1}{4}\cal N^2$. Since $U$ has no good asymptote, $\bar U$ is not a conic, hence $\Phi\neq \id$. The resolution $\Phi\colon (\ov{S}',D'+E')\rightarrow (\mathbb{P}^2,L_\8+\bar U)$ is minimal, so the last $(-1)$-curve produced by $\Phi$ is not a tip of $D'+E'$, hence it meets some twig of $D'+E'$. Let $W$ be a tip of one of these twigs. $W$ is not touched by $\Psi$ so, because $W^2<0$, it is contained in $\Supp \cal N$. Thus $\cal N\neq 0$ and we get $\varepsilon>-1$.

(iii) Suppose $\gamma\leq 0$. After blowing up over one of the points in $E\cap D$ we may assume that $E^2=0$. Then $U$ is a fiber of a $\C^*$-fibration of $\C^2$. The fibration is trivial over some Zariski open subset of the base, and hence the Euler characteristic of the total space over this subset vanishes. Thus, if $F\subseteq \C^2$ is the sum of the remaining fibers then $\chi(F)=\chi(\C^2)=1$. It follows that $F$ contains an irreducible (smooth) component with positive Euler characteristic. Since $\C^2$ contains no complete curves, it is necessarily $\C^1$. This is a good asymptote of $U$ and we reach a contradiction.

(iv) If $\gamma\neq 1$ then $D+E$ is snc-minimal. If $(\ov S,D+E)$ is not almost minimal then, since $D+E$ is connected and not negative definite, \cite[\S2.3]{Miyan-OpenSurf} implies that there exists a $\C^1$ contained in $S\setminus E$ witnessing the non-minimality. The latter is impossible by \ref{lem:no asymptote}.

(v) Suppose $\gamma\neq 1$. Let $\cal P$ be the positive part of the Zariski decomposition of $\ks+D+E$. By (iv) we compute $\cal P^2=2+e(D+E)-\varepsilon$. Since $\chi(S\setminus E)=1$, \ref{lem:BMY} gives (v). In case $\gamma=1$ we note that the snc-minimalization of $D+E$ does not touch the maximal twigs of $D+E$, so we get (v) by applying \ref{lem:BMY} to the resulting minimal model.
\end{proof}

Define $t_\lambda\in\{0,1\}$ by $t_{\lambda}=1$ if $p_h=1$ and $h>0$ and $t_{\lambda}=0$ otherwise. Define $t_{\ti \lambda}\in \{0,1\}$ analogously for $\ti \lambda$. Let \begin{equation}t=t_\lambda+t_{\ti \lambda}\in \{0,1,2\}\label{eq:def_of_t}.\end{equation}
We see easily that $t_{\lambda}=1$ if and only if $C$ is a $(-1)$-curve and $C$ together with some $(-2)$-twig of $D$ is contained in a twig of $D$. An analogous statement holds for $t_{\ti \lambda}$. The following inequality is proved in \cite[2.5]{KR-Some_properties_of_Cstar} as a consequence of the non-existence of a good asymptote.

\begin{proposition}\label{prop:basic_inequality}
$$2\varepsilon +\gamma\leq 7+t$$ \end{proposition}

\begin{corollary}\label{cor:gamma<9} $\gamma\leq 8$ and $h+j+\ti h+\ti j\leq 9+t-\varepsilon\leq 11$.   \end{corollary}

\begin{proof}
Suppose that $\gamma\geq 9$. By \ref{prop:basic_inequality} $\gamma=9$, $\varepsilon=0$ and $t=2$. By the above remark $D+E$ has at least two $(-2)$-twigs. It follows also that the branches $\lambda$, $\tilde\lambda$ are not simple, because otherwise $t\leq 1$. A maximal twig $[(2)_k]$ contributes to $e(D+E)$ by $\frac{k}{k+1}\geq\frac{1}{2}.$
 By \ref{lem:eps_and_gamma}(v) $e(D+E)\leq 1$, so  $D+E$ has two maximal twigs and they are both $(-2)$-curves. In particular $D$, whence $D'$, is a chain. The sequences of characteristic pairs for $\lambda$ and $\ti \lambda$ are $\binom{2}{2}_j,\binom{2}{1}$ and $\binom{2}{2}_{\ti j},\binom{2}{1}$ respectively. Then the sum $jc_1+\sum p_i+\ti j \ti c_1+\sum\ti p_i=2j+2\ti j+2$ is even, which contradicts \eqref{eq:1}, because $\gamma=9$. The second inequality follows directly from \ref{prop:basic_inequality} and \eqref{eq:h+hB}.
\end{proof}

\subsection{Branches are not simple} Recall that a branch $\lambda$ or $\ti \lambda$ is \emph{simple} if and only it meets $L_\8$ normally.

\bprop\label{prop:no_simple_branch} One can choose coordinates on $\C^2$ so that the branches $\lambda$, $\ti \lambda$ are separated and not simple. In particular, in these coordinates $c_1, \ti c_1>1$ and $h,\ti h\geq 1$.\eprop

\begin{proof} Suppose one of the branches is simple. We may assume it is $\ti \lambda$. Then $\lambda $ is not simple, because otherwise $\bar U$ is a conic, which clearly has a good asymptote.

Suppose $j>0$. After the first blowing-up over $\lambda\cap L_\infty$ the proper transform of $L_\infty$ becomes a $(0)$-curve and it meets the proper transform of $\bar U$ once. Let $L$ be the proper transform of a general member of the linear system of this $(0)$-curve on $\ov S$. Then $L\cdot D=1$ and $L\cdot E=1$, so $L\cap S$ is a good asymptote of $U$; a contradiction.

Thus $j=0$. The morphism $\Psi\colon\ov S'\rightarrow\ov S$ is a composition of blowdowns starting from the contraction of $L'_\infty$ if $L'_\infty$ is a $(-1)$-curve (otherwise $\Psi=\id$). Let $u$ be the number of these blowdowns. Then $-\gamma=-\gamma'+u$. The formulas \eqref{eq:1} and \eqref{eq:2} take the form
\begin{align*}
\gamma+u+2c_1+2&=p_1+\dots +p_h\\
\gamma+u+(c_1+1)^2&=p_1c_1+\dots+p_hc_h.
\end{align*}
Let $c_1=kc_2$, $p_1=k'c_2$. Let $c_1-p_1=\beta c_2$ i.e.\ $\beta=k-k'$. The numbers $k$ and $k'$ are relatively prime. We rewrite the formulas in the following form
\begin{align}
\gamma+u+2+(\beta+k)c_2&=p_2+\dots+p_h,\label{eq:1_simple_branch}\\
\gamma+u+1+\beta kc_2^2+2kc_2&=p_2c_2+\dots+p_hc_h.\label{2_simple_branch}
\end{align}
Multiply \eqref{eq:1_simple_branch} by $c_2$ and subtract \eqref{2_simple_branch}. We get \begin{equation} (\gamma+u+2-2k)c_2-\gamma-1-u+(k+\beta-\beta k)c_2^2=\sum\limits_{i\geq 2}p_i(c_2-c_i)\geq 0,\label{eq:3_simple_branch}\end{equation} hence, because $\gamma\geq 0$ by \ref{lem:eps_and_gamma}(iii), \begin{equation}(\gamma+u+2-2k)c_2 >((k-1)(\beta-1)-1)c_2^2\label{eq:4_simple_branch}.\end{equation} Then $\gamma+u+1-2k\geq ((k-1)(\beta-1)-1)c_2.$

Suppose that $\beta\geq 2$.  Then $k=k'+\beta\geq 3$ and the chain produced by $\binom{c_1}{p_1}$ which is contained between $L_\8'$ and the $(-1)$-curve created by the pair does not consist of only $(-2)$-curves. Hence it starts with $u-1$ $(-2)$-curves and then comes a $(\leq -3)$-curve. So the determinant of that chain, which is equal to $k'$, is at least $2u+1$. We obtain $2u\leq k'-1=k-\beta-1\leq k-3$, which gives $$\gamma> ((k-1)(\beta-1)-1)c_2+\frac{3k+1}{2}.$$ By \ref{cor:gamma<9} it follows that $8\geq \gamma> (k-2)c_2+\frac{1}{2}(3k+1)$, hence $k\leq 3$. Then $k=3$, so $c_2\leq 2$, $u=0$, $\beta=2$ and $k'=1$. It follows that $c_i=c_2=2$ for $i\geq 2$ or $h=1$. From  \eqref{eq:3_simple_branch} we get $(\gamma-4)c_2=\gamma+1+c_2^2$, which has no solution for $c_2\leq 2$ and $\gamma\leq 8$; a contradiction.

Thus $\beta=1$. Then the chain produced by $\binom{c_1}{p_1}$ which is contained between $L_\8'$ and the $(-1)$-curve created by the pair consists of $k'-1$ $(-2)$-curves. $\Psi$ contracts $L_\infty'$ and that subchain, so $u=k'=k-1$. Let $r$ be the number of pairs equal to $\binom{c_2}{c_2}$. Rewrite \eqref{eq:3_simple_branch} as:
$$c_2(\gamma-k+1)-\gamma-k+c_2^2=\sum\limits_{i\geq r+3}p_i(c_2-c_i).$$
Because $c_i<c_2$ for $i\geq r+3$, we have $c_2-c_i\geq\frac{c_2}{2}$. Then $c_2(\gamma-k+1)-\gamma-k+c_2^2\geq\frac{c_2}{2}\sum\limits_{i\geq r+3}p_i$. We get $\sum\limits_{i\geq r+3}p_i< 2(\gamma-k+1)+2c_2 $. From \eqref{eq:1_simple_branch} we get $\gamma+k+(k+1)c_2+1=rc_2+p_{r+2}+\sum\limits_{i\geq r+3}p_i< rc_2+p_{r+2}+2\gamma-2k+2+2c_2$.  Hence $$c_2(k-r-1)\leq p_{r+2}+\gamma-3k.$$

If $r\geq k'$ then, after performing $k'$ blowing ups of type $\binom{c_2}{c_2}$ and successive contractions starting from $L_\infty'$ we get new coordinates on $\C^2$ in which the branches of $\bar U$ are separated and none of them is simple. So we may assume that $r\leq k'-1=k-2$. By \eqref{eq:h+hB} $h=\gamma+k'+\varepsilon+2\geq \gamma+\varepsilon+r+3$, so since $\gamma\geq 0$, we have $h\geq r+3$. In particular, $c_{r+3}\geq 2$. The above inequality reads as $$c_2-p_{r+2}\leq\gamma-3k,$$ which gives $2\leq c_{r+3}\leq \gamma-3k\leq \gamma-6$. By \ref{prop:basic_inequality} it follows that $\gamma=8$ and $\varepsilon=0$, hence $k=2$, $k'=1$ and $c_i=2$ for $i\geq r+3$. Since $r\leq k-2$, we get also $r=0$, so $c_2>p_2$, hence $D+E$ has at least three tips. Since $k=2$ and $p_h=1$, at leat two of them are $(-2)$-tips, so $e(D+E)>1$. But $\varepsilon=0$, so we get a contradiction with \ref{lem:eps_and_gamma}(v).
\end{proof}

\subsection{Properties of the snc-minimalization $\Psi$} From now on we may, and shall, assume that both branches $\lambda$ and $\ti \lambda$ are not simple.

\begin{lemma}\label{lem:S-E_has_k=2} The surface $\ov S\setminus(D+E)=S\setminus U$ is of log general type. \end{lemma}

\begin{proof} The surface $S\setminus E$ has Euler characteristic $1$. Suppose it is $\C^1$- or $\C^*$-fibered. Then there is a line contained in some fiber (for a $\C^*$-fibration we argue as in \ref{lem:eps_and_gamma}(iii)). We may assume, blowing on $D$ if necessary, that the fibration extends to a $\PP^1$-fibration of $\ov S$. Let $L$ be the closure of the line and $F$ the fiber of the extension containing it. We have $E\cdot L\leq E\cdot F$. Since $S$ contains no completes curves, $E$ is not a $2$-section of the fibration. It follows that $E\cdot F\leq 1$, so $L$ is a good asymptote of $U$; a contradiction. Thus $S\setminus E$ is neither $\C^1$- nor $\C^*$-ruled. Because $S\setminus E$ contains no lines, the pair $(\ov S,D+E)$ becomes almost minimal after the snc-minimalization of $D+E$. By structure theorems for affine surfaces, if $S\setminus E$ is not of general type then $\kappa(S\setminus E)=0$ and by \cite[8.8]{Fujita-noncomplete_surfaces} the image of $D+E$, having arithmetic genus one, is a cycle of rational curves. Then $D+E$ is a cycle of rational curves. But,  since  both branches of $\bar U$ at infinity are not simple, $D+E$ has at least two tips. We reach a contradiction.
\end{proof}

\bcor\label{cor:eps=3-h0} $\varepsilon=3-h^0(2\ks+D+E)$.\ecor

\begin{proof} The Riemann-Roch theorem gives $\chi(\cal O_{\ov S}(2\ks+D+E))=\ks\cdot(\ks+D+E)+p_a(D+E)=3-\varepsilon$. Because $(K+D+E)^+$ is nef and (by \ref{lem:S-E_has_k=2}) big and because $(K+D+E)^-$ has proper fractional coefficients, the Kawamata-Viehweg vanishing theorem gives $\chi(\cal O_{\ov S}(2\ks+D+E))=h^0(2\ks+D+E)$.
\end{proof}

\bcor\label{lem:j<=1} $j\leq 1$. \ecor

\begin{proof} Suppose $j\geq 2$. Then $\ti j\geq 2$ and after blowing twice over each point of intersection of $\bar U$ with $L_\8$ the proper transform of $2L_\8$ together with the exceptional curves meeting it constitute a divisor $F_\8$ of type $[2,1,2]$ disjoint from the proper transform of $\bar U$. The linear system $|F_\8|$ gives a $\C^*$-fibration of $S\setminus U$, which contradicts \ref{lem:S-E_has_k=2}.
\end{proof}

If $j\neq 0$ then the line on $\PP^2$ tangent to $\lambda$ is different than $L_\8$. We denote its proper transform on $\ov S$ by $L_F$. We define $L_{\ti F}$ analogously.

\blem\label{lem:tangents} If $j\neq 0$ then $L_F$ meets the tip of $D'$ created by $\binom{c_1}{p_1}$. The tip and $L_F$ are not touched by $\Psi$. Moreover, $L_F\cdot E\geq 2$. \elem

\begin{proof} By \ref{lem:j<=1} $j=1$. Since the tangent to $\lambda$ on $\PP^2$ is a smooth curve, after blowing up twice over $\lambda\cap L_\8$ it separates from $\lambda$, which gives the first claim. The second claim follows from \ref{lem:Psi_properties}(a) below. Since $L_F\cap S\cong \C^1$, $L_F\cdot E\geq 2$ by \ref{lem:no asymptote}.
\end{proof}

\begin{notation}\label{not:C_and_Q} Recall that the curves $C'$ and $\ti C'$ (defined in \ref{ssec:HN}) are different. We put $C=\Psi_*(C')$, $\ti C=\Psi_*(\ti C')$, where $\Psi\: (\ov{S}',D'+E')\to (\ov{S},D+E),$ is as in \ref{not:setup}, and we define $Q_1$ (respectively $\ti Q_1$) as the connected component of $D-C$ (respectively of $D-\ti C$) which does not contain $\ti C$ (respectively $C$). Put $Q_0=D-C-\ti C-Q_1-\ti Q_1$. It follows that $Q_1$ and $\ti Q_1$ are chains. They are the maximal twigs of $D+E$ created by the pairs $\binom{c_h}{p_h}$, $\binom{\tilde c_h}{\tilde p_h}$ (nonzero, because the log resolution $\Phi$ is minimal). We denote by $G$ and $\ti G$ the components of $Q_0$ meeting $C$ and $\ti C$ respectively (see Fig. \ref{fig:Sbarprim}.) Let $T_1'$ (resp. $\tilde T_1'$) be the  branching component of $D'+E'$ contained in $F$ (resp. $\tilde F$) nearest to $L'_\infty$. These are the last curves produced by the pairs $\binom{c_{1}}{p_{1}}$ and $\binom{\ti c_{1}}{\ti p_{1}}$ respectively (note in case $h=1$ we have $T_1'=C_1'$; similarly for $\ti T_1'$.) Put $T_1=\Psi_*T_1'$ and $\ti T_1=\Psi_*\ti T_1'$.
\end{notation}

\begin{lemma}\label{lem:Psi_properties}\
	\begin{enumerate}[(a)]
    \item $\Psi$ involves only contractions which are inner for $D'+E'$ and it does not touch maximal twigs of $D'+E'$.
    \item $C'$, $\ti C'$ and $E'$ are not touched by $\Psi\colon\ov S'\To\ov S$. In particular, $\gamma'=\gamma$.
    \item $C$ and $\wt C$ meet different components of $Q_0$, i.e.\ $G\neq \wt G$.
    \item $Q_0$ is snc-minimal and contains only curves of negative self-intersection.
	\end{enumerate}
\end{lemma}

\begin{proof}
(a) There is a unique chain in $D'$ meeting both $T_1'$ and $\tilde T'_1$ and not containing them. It contains $L'_\infty$ and its components are non-branching in $D'+E'$. All contractions in $\Psi$ take place inside this chain. The statement follows.

(b) Since neither $\lambda$ nor $\tilde \lambda$ is simple, $E'$ is disjoint from $L'_\infty$. Suppose $C'+\ti C'$ is touched by $\Psi$ at some stage (this may happen only if $h=1$ or $\ti h=1$). We have $\Psi\neq \id$, so $L_\8'$ is a $(-1)$-curve. Since the branches are not simple, this implies that $j>0$, hece $j=1$ by \ref{lem:j<=1}. At some stage of $\Psi$ the proper transform of $C'$ or of $\ti C'$ becomes a $0$-curve and its total transform on  $\ov S'$ induces a $\PP^1$-fibration $\pi$ of $\ov S'$ with $E'$  as a $1$-section. Then $L_F$ is contained in a fiber, so it is met by $E$ at most once, hence it is a good asymptote; a contradiction. Therefore, $C'+\ti C'$ is not touched by $\Psi$. It follows that $E'$ is not touched.

(c) Suppose $G=\wt G$. Since by (b) $C'+\ti C$ is not touched, the proper transform of $G$ in $\ov S'$ is a component of $D'$ meeting both $C'$ and $\ti C'$. It must be $L'_\infty$, so $h=\ti h=1$ and $p_1=\ti p_1=1$. But then $\ov U$ is smooth, so it is a conic. We reach a contradiction, because in the latter case there exists a good asymptote of $U$.

(d) Suppose $Q_0$ contains a non-branching $(-1)$-curve $D_0$. Because, possibly with the exception of $E$, the divisor $D+E$ contains no non-branching $(-1)$-curves, $D_0$ is a branching component of $D+E$. Thus $D_0$ meets $C$ or $\ti C$, say $C$. We have $C\cdot \ti C=0$, so $D_0\neq \ti C$.  Also, $\Psi\neq \id$, so $j\neq 0$. By (b) $C$ and $\ti C$ are $(-1)$-curves. By \ref{lem:tangents} $L_F\cdot (C+D_0)=0$ and $L_F\cdot E\geq 2$. The former equality implies that $L_F$ is contained in a fiber of the $\PP^1$-fibration of $\ov S$ given by the linear system $|C+D_0|$. Then $L_F\cdot E\leq (C+D_0)\cdot E=C\cdot E=1$; a contradiction.

Suppose $Q_0$ contains a component $D_0$ of non-negative self-intersection. Then $D_0$ is, by the definition of $D+E$, contained in the chain $T_1+R+\ti T_1$, where $T_1=\Psi_*T_1'$, $\ti T_1=\Psi_*\ti T_1'$ and $R$ is the chain in $D$ between $T_1$ and $\ti T_1$. Also, $\Psi\neq \id$, so $j>0$, hence $L_F$ is contained in a member of the linear system $|D_0|$. Then $L_F\cdot E\leq D_0\cdot E\leq 1$; a contradiction with \ref{lem:tangents}.
\end{proof}

\blem \label{prop:gamma>1} $\gamma\geq 2$.
\elem

\begin{proof}Suppose that $\gamma=1$, i.e.\ $E^2=-1$.  Let $T_1$ and $\ti T_1$ be as in \ref{not:C_and_Q}. The divisor $F=C+E$ induces a $\PP^1$-fibration of $\ov S$ for which $D$ has three horizontal components: $G$, $\ti C$ and $Q_1^+\subseteq Q_1$. Suppose that $A$ is a $(-1)$-curve in $D-C-\ti C$. Then $A$ is contained in $Q_0$. By \ref{lem:Psi_properties}(d) $A$ is a branching component of $Q_0$. Also, $\Psi\neq \id$, so $j\neq 0$. If $A\cdot C=1$ then the linear system $|C+A|$ induces a $\PP^1$-fibration of $\ov S$ for which $E$ is a section, so $L_F$ if vertical and $L_F\cdot E\leq 1$, which contradicts \ref{lem:no asymptote}. Therefore, $A\cdot C=0$. Let $F_A$ be the fiber containing $A$. Since fibers of $\PP^1$-fibrations do not contain branching (in fiber) $(-1)$-curves, $C$ meets one of the adjacent components of $A$ in $D$, which is therefore a section of the fibration. It follows that multiplicity of $A$ in $F_A$ equals $1$. But the remaining two adjacent components meet $A$ and are contained in $F_A$, hence the multiplicity of $A$ is greater than $1$; a contradiction.

Thus $D-C-\ti C$ contains no $(-1)$-curves, so the only vertical $(-1)$-curve in $D$ is $C$. It follows that there is no fiber contained in $D$. Indeed, such a fiber would have to be smooth, and by \ref{lem:Psi_properties}(d) we know that $D$ contains no $(0)$-curves. Let $F_1$ be the fiber containing $\ti G$ (note $G\neq \ti G$ by \ref{lem:Psi_properties}(c)). By \ref{lem:rulings} $\Sigma_S=1$, i.e., there exists exactly one fiber $F_2$ which contains more than one, and in fact two, components not contained in $D$. Any other fiber contains a unique component not contained in $D$.  It may happen that $F_1=F_2$. Suppose there is a singular fiber $F_0$ other than $E+C$, $F_1$ and $F_2$. The unique component $L_0\subseteq F_0$ not contained in $D$ is also the unique $(-1)$-curve in $F_0$. Since $L_0$ has multiplicity bigger than $1$ and since $F_0-L_0$ has at most two connected components (both contained in $D$), two of the sections contained in $D$ meet a common connected component of $F_0-L_0$. Since $C$ is not a component of $F_0$, the sections are necessarily $G$ and $\ti C$, hence $F_0$ contains $\ti G$; a contradiction. Since the three sections contained in $D$ are disjoint, it is easy to see that we may contract successively all $(-1)$-curves in $E+C+F_1+F_2$ in such a way that the images of these sections, call them $H_1$, $H_2$, $H_3$, remain disjoint. Because there are no other singular fibers, this results with a morphism $\ov S\to \F$ onto a Hirzebruch surface $\F$. Since $H_2-H_3$ intersects trivially with $H_1$ and with a fiber, it is numerically trivial, hence $H_2^2=H_2\cdot H_3=0$. Then $\F=\PP^1\times \PP^1$, so $S\setminus (E+C+F_1+F_2)$ is isomorphic to $\PP^1\times \PP^1$ with three fibers and two disjoint sections (images of $G$ and $Q_1^+$) removed, i.e.\ to $\C^{**}\times \C^*$. But this means that $S\setminus E=S\setminus U$ contains an open subset with Kodaira dimension equal to $\kappa(\C^{**}\times \C^*)=1$, which contradicts \ref{lem:S-E_has_k=2}.
\end{proof}

\section{\texorpdfstring{Surgeries on $(\C^2,U)$}{Surgeries on (C2,U)}}\label{sec:surgeries}

In this section we analyse surfaces resulting from surgeries on $S=\C^2$. We cut out $U$ and we glue in $C$, $\ti C$ or both. Studying the geometry of resulting surfaces we obtain a lot of information on the $\C^*$-embedding $U\mono S=\C^2$. In particular, we improve the upper bound on $\gamma$ to $\gamma\leq 5$.

\subsection{Double sided surgery - the surface $Y$}

By \ref{lem:Psi_properties}(a) $\Psi$ does not touch maximal twigs of $D'+E'$. Recall that $Q_1$ and $\wt Q_1$ are the maximal twigs of $D+E$ which are images of maximal twigs of $D'+E'$ produced by the pairs $\binom{c_h}{p_h}$ and $\binom{\ti c_{\ti h}}{\ti p_{\ti h}}$ respectively (see \ref{not:C_and_Q}). We have $D=Q_1+C+Q_0+\ti C+\ti Q_1$. Put $Q=Q_1+Q_0+\ti Q_1+E$ and $Y=\ov S\setminus Q.$  This surface is obtained from $S=\C^2$ by cutting out $E$ and gluing in $C+\ti C$. We have $\chi(Y)= -1$.

\begin{lemma}\label{lem:k(Y)} If $\gamma=2$ then $\kappa(Y)=-\8$. If $\gamma+t\geq 6$ (or, more generally, if $\varepsilon\leq t$) then $2\ks+Q\geq 0$, so $\kappa(Y)\geq 0$. \end{lemma}

\begin{proof} If $\gamma=2$ then $nE$ is in the fixed part of a divisor $n(\ks+Q)$, so $\kappa(Y)=\kappa(\ks+Q_0+Q_1+\ti Q_1)=\kappa(\ks+D-C-\wt C)\leq \kappa(\ks+D)=-\8$. Note that if $\gamma+t\geq 6$ then \ref{prop:basic_inequality} gives $\varepsilon\leq t+\frac{1}{2}$. Thus for the proof of the second part of the lemma we may assume $\varepsilon\leq t$. By \ref{lem:Psi_properties}(b) $C$ and $\wt C$ are $(-1)$-curves, so $$\ks\cdot(\ks+Q)=\ks\cdot(\ks+D+E)-\ks\cdot C-\ks\cdot \ti C=4-\varepsilon.$$ We have $\varepsilon\leq 2$, so by \ref{cor:eps=3-h0} $2\ks+D+E\geq 0$. Suppose $h^0(2\ks+Q)=0$. If $\varepsilon=0$ then the Riemann-Roch theorem gives $h^0(2\ks+Q)+h^0(-\ks-Q)\geq \ks\cdot(\ks+Q)+p_a(Q)=1,$ so $\ks+C+\wt C=(2\ks+D+E)+(-\ks-Q)\geq 0$ and hence $\ks\geq 0$, which is impossible.

Suppose $\varepsilon=1$. Then $\ks\cdot(\ks+Q)=3$ and $t\geq 1$, so $Q_1$ or $\ti Q_1$, say $\ti Q_1$, consists of $(-2)$-curves. The Riemann-Roch theorem gives $h^0(-\ks-Q_0-Q_1-E)+h^0(2\ks+Q_0+Q_1+E)> 0,$ so $\ks+C+\ti C+\ti Q_1=2\ks+D+E+(-\ks-Q_0-Q_1-E)\geq 0,$ hence $\ks\geq 0$; a contradiction.

Thus $\varepsilon=2$. Then $\ks\cdot(\ks+Q)=2$ and $t=2$, so $Q_1$ and $\ti Q_1$ consist of $(-2)$-curves. Now $h^0(-\ks-Q_0-E)+h^0(2\ks+Q_0+E)>0,$ so $\ks+C+\ti C+\ti Q_1+Q_1=2\ks+D+E+(-\ks-Q_0-E)\geq 0,$ hence $\ks\geq 0$; a contradiction.\end{proof}

\begin{lemma}\label{lem:log-exc_curves_on_Y} Assume $Y$ has no $\C^1$-fibration. Then there is no $(-1)$-curve $L$ on $\ov S$ such that $L\cdot Q_0=0$, $L$ meets two connected components of $Q_1+E+\ti Q_1$ and together with these components contracts to a quotient singularity. \end{lemma}

\begin{proof} Suppose that such an $L$ exists. Let $\pi\colon \ov S\To \ov X$ be the contraction of $L$ and the two connected components of $Q_1+E+\ti Q_1$ to the point $q_1\in \ov X$ (the possibility that $q_1$ is smooth is not excluded). Let $Q_2$ be the third connected component of $Q_1+E+\ti Q_1$. The surface $X=\ov S\setminus Q_0$ is a sum of the open topological subspaces $S$ and the sum of tubular neighbourhoods of $C+Q_1$ and $\ti C+\ti Q_1$, which are all simply connected. By the van Kampen theorem $X$ is simply connected. Let $\ov X\To \ov X'$ be the contraction of $Q_2$ to a cyclic singular point $q_2\in X'$. Then $X'=\ov X'\setminus Q_0$ is simply connected and we have $\rho(\ov X')=\#Q_0$. Since the components of $D$ are independent in $\NS(\ov S)$, the components of $Q_0$ are independent in $\NS(\ov S)$, hence the components of $Q_0$ are independent in $\NS(\ov X')$. It follows that they generate $\NS(\ov X')$, so $X'$ is affine by an argument by Fujita \cite[2.4(3)]{Fujita-noncomplete_surfaces}. It follows that $b_3(X')=b_4(X')=0$ and $H_2(X',\Z)$ has no torsion. Since $\chi(X')=1$, we have $b_2(X')=b_1(X')=0$, so $X'$ is simply connected and $\Z$-acyclic, hence contractible. Moreover $\kappa(X')=\kappa(\ov S\setminus Q_0)\leq \kappa(S)=-\8$. Because $q_2\in X'$ is a non-trivial cyclic singularity, by \cite[1.1, 3.1]{KR-Contractible_surfaces} $X'\setminus \Sing X'$ has a $\C^1$-fibration. Then $Y$ has a $\C^1$-fibration; a contradiction.
\end{proof}

\bprop\label{lem:Y minimal} If $Y$ has no $\C^1$-fibration then the pair $(\ov S, Q)$ is almost minimal.\eprop

\begin{proof} By  \ref{prop:gamma>1} and \ref{lem:Psi_properties} $Q$ snc-minimal. Let $(\ov{Y}',T')$  be an almost minimal model of $(\ov S,Q)$. $\ov{Y}'$ is obtained from $\ov{S}$ by a sequence of birational morphisms $$\pi_i : \ol{Y}_i \to \ol{Y}_{i+1},\ \  \ol{S}=\ol{Y}_0 \to \ol{Y}_1 \to \cdots \to \ol{Y}_\ell=\ol{Y}'.$$ Let $T_i=(\pi_{i-1})_*(T_{i-1})$, $T_0=Q$, $T'=T_\ell$. Let $Y_i=\ol{Y}_i\setminus T_i$. For every $i$ there exists a $(-1)$-curve $C_i \nsubseteq T_i$ such that $\pi_i : \ol{Y}_i \to \ol{Y}_{i+1}$ is the   {\em snc-minimalization} of $C_i+T_i$. Finally, for the almost minimal model $(\ol{Y}',T')$, the negative part $(K_{\ol{Y}'}+T')^-$ coincides with the bark $\Bk T'$ if $\kappa(Y)\geq 0$. The contractions in this process involve only  curves (or their images) contained in the support of $(\ks+Q)^-$.  We put
$$e(\ol{Y_i},T_i)=\chi(\ol{Y_i}\setminus T_i)+\#\{\mbox{connected components of $T_i$}\}.$$ Relying on the theory of peeling \cite[2.3.6]{Miyan-OpenSurf}, which gives a description of curves $C_i$, we find that $e(\ov Y_{i+1},T_{i+1})=e(\ov Y_i,T_i)-1.$ Hence $e(\ov Y',T')=e(\ov S,Q)-\ell=3-\ell.$

Suppose that $(\ov S, Q)$ is not almost minimal, i.e.\ $\ell\geq 1$. If all connected components of $T'$ are of quotient type then the intersection matrix of $Q+C_0$ is negative definite and has rank $\#Q+1=\rho(\ov Y')$, which contradicts \ref{lem:Hodge lemma}. Thus $T'$ contains a connected component which is not of quotient type.

Suppose $\kappa(Y)=-\8$. There does not exist a $\PP^1$-fibration of $Y$, because then $E$ would be a smooth fiber, and this is impossible by \ref{lem:eps_and_gamma}(iii). Because $Y$ has no $\C^1$-fibration, by \cite{MiTs-PlatFibr} $T'$ consists of two disjoint forks, exactly one of which is of quotient type and $Y'\cong(\C^2-\{0\})/ G$, where $G$ is a finite group. Hence $\chi(Y')=0$ and therefore $e(\ov Y',T')=2$.  It follows that  $\ell=1$. But then two connected components of $Q$ do not meet $C_0$, hence one of the connected components of $T'$ is a chain; a contradiction.

Thus $\kappa(Y)\geq 0$. Let $k$ be the number of connected components of $T'$ which are of quotient type and let $G_j$ be the local fundamental groups. The resulting quotient points are singular because $T'$ is snc-minimal, hence $|G_j|\geq 2$. Let $u$ be the number of connected components of $T'$. We have $\ell\geq 1$ and, by the argument above, $k\leq u-1$. By \ref{lem:BMY} $$\chi(Y')+\frac{k}{2}\geq \chi(Y')+\sum\limits_{i=1}^k \frac{1}{|G_j|}\geq 0.$$ We have $\chi(Y')=e(\ov Y',T')-u=3-\ell-u$. We obtain $3-\ell -u +\frac{k}{2}\geq 0,$ so $3\geq \ell+u-\frac{k}{2}\geq \ell +1+\frac{k}{2}$, hence $k\leq 2$ and $u\leq 3$. Suppose $\chi(Y_{i+1})>\chi(Y_i)$ for some $i$. This is possible only if $C_i$ meets two connected components of $T_i$ and contracts to a smooth point together with these connected components. By \ref{lem:log-exc_curves_on_Y} one of these connected components contains the image of $Q_0$. But then all connected components of $T'$ are of quotient type; a contradiction.

It follows that $\chi(Y')\leq \chi(Y)=-1$, hence $k=2$. Then $|G_1|=|G_2|=2$, $\ell=1$ and $u=3$. Also, $\chi(Y')=-1=\chi(Y)$. The latter implies that $C_0$ meets two connected components of $Q$. Suppose that $C_0\cdot Q_0=0$. Then, since $Q_1+E+\ti Q_1+C_0$ is contained in $\Supp (\ks+Q)^-$, the intersection matrix of $Q_1+E+\ti Q_1+C_0$ is negative definite, which implies that the intersection matrix of $Q_0$ is not negative definite. Hence $Q_0$ is not a contractible connected component of $T'$. Since $k=2$, $C_0$ together with two connected components of $Q_1+E+\ti Q_1$ contracts to a quotient singularity; it is impossible by \ref{lem:log-exc_curves_on_Y}. Thus $C_0$ meets $Q_0$ and one of connected components of $Q_1+E+\ti Q_1$. Since $|G_1|=|G_2|=2$, the remaining two connected components are $(-2)$-curves. By \ref{lem:BMY} $((K_{\ov Y'}+T')^+)^2=0$, so $\kappa(Y)\leq 1$. We have $\gamma\neq 2$, otherwise $\kappa(Y)=-\8$ by \ref{lem:k(Y)}.

Thus $C_0$ meets $Q_0$ and $E$ and we have $d(Q_1)=c_h=2$, $d(\ti Q_1)=\ti c_{\ti h} =2$. By \eqref{eq:1} and \eqref{eq:2} $4$ divides $d^2-\gamma$ and $2$ divides $d$, hence $4$ divides $\gamma$. Suppose that $\gamma=8$. Then $\varepsilon=0$ and by  \ref{lem:eps_and_gamma}(v)  $D+E$ has two $(-2)$-tips $Q_1, \ti Q_1$ and no other tip. Hence $D$ is a chain  so $Q_0$ is a chain. Then $Q_0$ is contained in $\Supp(\ks+Q)^-$ and $Q+C_0$ is contained in $\Supp(\ks+Q)^-$, hence the intersection matrix of $Q+C_0$ is negative definite, which again contradicts \ref{lem:Hodge lemma}.

Thus $\gamma=4$. By \ref{prop:basic_inequality} $\varepsilon\leq 2$. We have $t=2$, so by \ref{lem:k(Y)} $2\ks+Q\geq 0$. Also, $Q_0+C_0+E$ is not of quotient type and $\kappa(Y)=0,1$. We write $T'=Z+Q_1+\ti Q_1$ where $Z$ is the image of $Q_0+C_0+E$ in $\ov Y'$. Note that since $Z$ is not negative definite, it cannot be a chain. Indeed, otherwise it contains a component of non-negative self-intersection, which would imply that $Y'$, and hence $Y$, is $\PP^1$- or $\C^1$-fibered. Let $\cal P=(K_{\ov Y'}+T')^+$. We have $2K_{\ov Y'}+T'\geq 0$ and $\cal P\cdot (K_{\ov Y'}+T')=\cal P^2=0$. Since $2K_{\ov Y'}+2T'\geq T'$, we obtain $0\geq \cal P\cdot T'$, i.e.\ $\cal P\cdot V=0$ for every irreducible component $V$ of $T'$. By \cite[8.8]{Fujita-noncomplete_surfaces} either $Z$ is a fork $H+R_1+R_2+R_3$ where $H$ is a branching component and $R_1,R_2, R_3$ are the maximal twigs or $Z$ is a rational tree which consists of a nonzero chain and two $(-2)$-curves attached to each tip of the chain (four $(-2)$-curves in total). We have $K_{\ov Y'}\cdot (K_{\ov Y'}+T')=\ks\cdot(\ks+Q+C_0)+m=3-\varepsilon+m\geq m+1$, where $m$ is the number of outer (with respect to $Q_0+C_0+E$) blow-downs in $\ov S\rightarrow\ov Y'$. From this $-4+\Bk^2 Z=\Bk^2 T'=(K_{\ov Y'}+T')^2=m-\varepsilon-3\geq m-5$, hence $\Bk^2 Z\geq m-1$. Then $m=0$ and $0>\Bk^2 Z=1-\varepsilon$, hence $\varepsilon=2$ and $\Bk^2 Z=-1$. But then $Z$ cannot have four $(-2)$-tips, hence $Z$ is a fork. From the description of the fork in loc. cit. we get $\frac{1}{d(R_1)}+\frac{1}{d(R_2)}+\frac{1}{d(R_3)}=1$, so $(d(R_1), d(R_2), d(R_3))=(3,3,3)$ or $(2,4,4)$ or $(2,3,6)$. Since $\Bk^2 Z=-1$ one checks easily that the maximal twigs of $Z$ are tips, so $\#Z=4$. We have $\rho(\ov S)=\#(Q_0+C_0)$, so $\rho(\ov Y')=\#T'=6$, hence $K_{\ov Y'}^2=4$ by the Noether formula. Because $K_{\ov Y'}\cdot (K_{\ov Y'}+T')=1$, we get $K_{\ov Y'}\cdot Z=K_{\ov Y'}\cdot T'=-3.$ Since $\sum\limits_{i=1}^3K_{\ov Y'}\cdot R_i>0$ we get $K_{\ov Y'}\cdot H\leq -4$, hence $H^2\geq 2$.  Now since $m=0$, the blowing ups in the reverse of $\ov S\rightarrow \ov Y'$ are inner, i.e.\ we blow up over $H\cap R_i$ for some $i$. Moreover, we blow up only once on $H$, because the connected component of $Q$ containing $E$ consists of $E$ only. So the proper transform of $H$ in $\ov S$ is a positive curve contained in $D$; a contradiction by \ref{lem:Psi_properties}(d).
 \end{proof}

\blem\label{lem:Y_geometry} If $\kappa(Y)=-\8$ then $Q_0$ is of quotient type or $Y$ has a $\C^1$-ruling with no base points on $\ov S$, $Q_0$ is branched and has a maximal twig of type $[(2)_{\gamma-1}]$. If $\kappa(Y)\geq 0$ then $$\frac{1}{\ti c_h}+\frac{1}{\ti c_{\ti h}}+\frac{1}{\gamma}+\frac{1}{\Gamma(Q_0)}\geq 1.$$  \elem

\begin{proof} If $\kappa(Y)\geq 0$ then $((K_{\ov S}+Q)^+)^2\geq 0$, so the above inequality follows from \ref{lem:BMY}. Assume $\kappa(Y)=-\8$ and $Y$ is not $\C^1$-ruled. If $Y$ is $\PP^1$-ruled then by \ref{lem:rulings} $\nu=\Sigma_Y+1\geq 1$, so $Q$, being snc-minimal, contains a $0$-curve. But the latter is impossible by \ref{lem:Psi_properties}(d), hence $Y$ is not $\PP^1$-ruled. By \ref{lem:Y minimal} $(\ov S,Q)$ is almost minimal, so after the contraction of connected components of $Q$ which are of quotient type it becomes a log del Pezzo surface of rank one. Since $Q$ has more than two connected components, by \cite[2.5.1]{Miyan-OpenSurf} the resulting del Pezzo is closed, i.e.\ all connected components of $Q$ are of quotient type. Assume $p\:Y\to B$ is a $\C^1$-fibration. Again, $\nu=0$, otherwise $Q$ would contain a $0$-curve. Let $(\ov S,\ti Q)\to (\ti S,Q)$ be a minimal modification over $Q$ such that $p$ has no base points of $\ov S$. We have $\Sigma_Y=\nu=0$, so every singular fiber has a unique component not contained in $\ti Q$, and hence, by the minimality, a unique $(-1)$-curve. The complement of this $(-1)$-curve in the fiber has at most two connected components, one of which meets the section contained in $\ti Q$. Since $\ti Q$ has four connected components, there are at least three singular fibers. This is possible only if $\ti S=S$ and the section contained in $Q$ is in fact contained in $Q_0$. In particular, $p$ has no base points on $\ov S$. Let $F_E$ be the fiber containing $E$. Since $E$ is irreducible, we have necessarily $F_E=[(2)_{\gamma-1},1,\gamma]$, so $Q_0$ contains a maximal twig of type $[(2)_{\gamma-1}]$.
\end{proof}

While in principle the divisor $Q_0$ my be complicated, this is not so if $\gamma\geq 6$.

\bcor \label{lem:Q0_quotient_for_gamma>=6} If $\gamma\geq 6$ then $Q_0$ is of quotient type. \ecor

\begin{proof} Suppose $Q_0$ is not of quotient type. From \ref{lem:k(Y)} we get that $\ovk(Y)\geq 0.$ By \ref{lem:Y minimal} $(\ov S, Q)$ is almost minimal. The log BMY inequality gives $1\leq \frac{1}{c_h}+\frac{1}{\ti c_h}+\frac{1}{\gamma}$ and the inequality is strict if $\ovk(Y)=2$.  Suppose $c_h=\ti c_{\ti h}=2$. Equations \eqref{eq:1} and \eqref{eq:2} imply that $4$ divides $\gamma$, hence $\gamma=8$. Then $\varepsilon=0$ and $e(D+E)\leq 1$ by \ref{lem:eps_and_gamma}(v). Since $Q_1$ and $\ti Q_1$ are $(-2)$-tips of $D+E$, $D+E$ has no more tips, so $D$ is a chain. But then $Q_0$ is a chain, hence of quotient type; a contradiction. Therefore $\max\{c_h,\ti c_{\ti h}\}>2$. Since $\gamma\geq 6$, the log BMY is an equality, so $\{c_h, \ti c_h\}=\{2,3\}$, $\gamma=6$ and $\cal P^2=0$. In particular, $\ovk(Y)=0,1$. Say $\ti c_2=2$. Then $\Bk^2 \ti Q_1=-2$ and $d(Q_1)=c_h=3$.

By \ref{lem:k(Y)} $2\ks+Q\geq 0$. We put $\cal{P}=(\ks+Q)^+$ and we argue as in the proof of \ref{lem:Y minimal} that $\cal{P}\cdot Q=0$ and hence that $Q_0$ is a fork with $\delta=1$ or a tree with exactly four $(-2)$-tips. Assume the latter case. Then $\Bk^2 Q_0=-2$.  We have $-4-\varepsilon=(\ks+Q)^2=\Bk^2 Q=-4-\frac{4}{6}+\Bk^2 Q_1$, i.e.\ $\varepsilon=-\Bk^2 Q_1+\frac{2}{3}.$ Now, since $d(Q_1)=3$, $Q_1$ is either a $(-3)$-curve or a chain of two $(-2)$-curves, so $\varepsilon=2$ or $\varepsilon=2+\frac{2}{3}$. But if $\varepsilon=2$ then $\gamma\leq 5$ by \ref{prop:basic_inequality}; a contradiction.

Thus, $Q_0$ is a fork with a branching component $H$ and three maximal twigs $R_1$, $R_2$, $R_3$, such that $\delta(Q_0)=\frac{1}{d(R_1)}+\frac{1}{d(R_2)}+\frac{1}{d(R_3)}=1$. We have $-4-\varepsilon=(\ks+Q)^2=\Bk^2 Q= \Bk^2 Q_1-2-\frac{2}{3}-e(Q_0)$. By \ref{prop:basic_inequality} $\varepsilon=1$, so $1\leq \delta(Q_0)\leq e(Q_0)=\frac{7}{3}-e(Q_1)$. If $Q_1=[2,2]$ then $e(Q_1)=2$ and the latter inequality fails. Therefore, $Q_1$ is a $(-3)$-curve, hence $e(Q_1)=\frac{4}{3}$ and $\delta(Q_0)=e(Q_0)=1$. It follows that the twigs of $Q_0$ are irreducible.  Then $b_2(\ov S)=\#Q+1= 8$, so $\ks^2=2$ by the Noether formula. Since $\ks\cdot (\ks+Q)=4-\varepsilon=3$ we get $\ks\cdot Q=1$. We compute $\ks\cdot H=1-\ks\cdot \sum R_i-\ks\cdot Q_1-\ks\cdot \ti Q_1=-\ks\cdot \sum R_i$. But $\ks\cdot\sum R_i\geq 3$, so $\ks\cdot H\leq -3$, i.e.\ $H^2\geq 1$; a contradiction  in view of \ref{lem:Psi_properties}(d).
\end{proof}

\subsection{One-sided surgeries - surfaces $Y_C$ and $Y_{\ti C}$}

By \ref{cor:gamma<9} $\gamma\leq 8$. The goal of this subsection is to improve this bound to $\gamma\leq 5$. We suppose, for a contradiction, that $\gamma\geq 6$. By \ref{prop:basic_inequality} $\varepsilon \leq 1$. We introduce some notation. By $ Y_C$ (resp. $ Y_{\tilde C}$)  we denote the surface $\ov S\setminus (D+E-C)$ (resp. $\ov S\setminus (D+E-\tilde C)$). So $Y_C\setminus C=Y_{\ti C}\setminus {\tilde C}=Y$. The boundary divisor of $Y_C$ equals $Q+\tilde C$ and the boundary divisor of $Y_{\tilde C}$ equals $Q+\tilde C$. Let $Y_C'$ (resp. $Y_{\ti C}'$) be the surface obtained from $Y_C$ (resp. $Y_{\ti C}$) by contracting $Q_1$ (resp. $\ti Q_1$) to a (quotient) singular point $z$ (resp. $z'$). Recall that $d(Q_1)=c_h$ and $d(\ti Q_1)=\ti c_{\ti h}$, see \ref{rem:capacity_using_pairs}.

\blem\label{lem:no_simple_fibers_in_D}\ \benum[(a)]

    \item $D$ does not contain a component $D_0$ for which $D_0^2\geq 0$.

    \item $D$ does not contain components $D_1$, $D_2$ for which $D_1^2=D_2^2=-1$ and $D_1\cdot D_2=1$.
\eenum\elem

\begin{proof} If $j=0$ then $D'=D$ is snc-minimal and both statements are true. So by \ref{lem:j<=1} we may assume $j=1$. Suppose (a) or (b) fails. In case (a) blow up on $D_0$ until it becomes a $0$-curve; denote it by $F_\8$. In case (b) put $F_\8=D_1+D_2$. In both cases $L_F\cdot F_\8=0$ and $E\cdot F_\8\leq 1$ (note that $E$ meets $D$ only in $C$ or $\ti C$ and we have $C\cdot \ti C=0$). Thus $L_F$ is contained in a fiber of the $\PP^1$-fibration given by $|F_\8|$ and $E\cdot L_F\leq E\cdot F_\8\leq 1$. By \ref{lem:tangents} this is a contradiction.
\end{proof}

\blem With the above notation we have:\benum[(i)]

    \item $\kappa(Y_C'\setminus\{z\}) \geq 0$.

    \item The pair $(\ov S, Q+\tilde C)$ is almost minimal.

\eenum Analogous statements hold for $Y_{\tilde C}$ and $(\ov S,Q+C)$.
\elem

\begin{proof} (i) Suppose that $\kappa(Y'_C\setminus\{z\})=-\infty$. First of all note that $\ti C$ is a branching $(-1)$-curve in $Q_0+\ti C+\ti Q_1+E$, which implies that the latter divisor cannot be vertical for a $\PP^1$-fibration of $\ov S$. It follows that $Y'_C\setminus\{z\}$ is not $\PP^1$-ruled. Suppose it is also not $\C^1$-ruled. Since $\chi(Y_C')=1$, by the result of Miyanishi-Tsunoda \cite{MiTs-PlatFibr} $Y_C'$ is isomorphic to $\C^2// G$, where $G$ is a small non-abelian subgroup in $GL(2,\C)$. But $z$ is a cyclic quotient singularity, so we come to a contradiction. So there is a $\C^1$-fibration $g\: Y_C\setminus Q_1\To \PP^1$. Let $X\to \ov S$ be a minimal modification over $Q+\ti C$, such that $g$ extends to a $\PP^1$-fibration $G\: X\To \PP^1$. From Lemma \ref{lem:rulings} we get $\nu>0$. Let $G_0$ be a fiber contained in $X\setminus Y_C$. If $G_0$ is irreducible then the image of $G_0$ in $Q+\ti C$ is a nonnegative curve; a contradiction by \ref{lem:no_simple_fibers_in_D}. Hence $G_0$ contains a $(-1)$-curve $B'$, which, by the minimality of the modification, is a branching component in $X\setminus Y_C$. Since $B'$ is not a branching component of $G_0$, it meets a horizontal component of the boundary. But the latter is a section, so the multiplicity of $B'$ in $G_0$ equals $1$ and hence $B'$ meets only one component of $G_0$. Then $B'$ is not a branching component in $X\setminus Y_C$; a contradiction.

(ii) Put $Q_2=Q_0+\ti C+\ti Q_1+E=Q+\ti C-Q_1$. Since $\#(Q_1+Q_2)=\rho(\ov S)$, \ref{lem:Hodge lemma} implies that the intersection matrix of $Q_2$ is not negative definite. By \ref{lem:Psi_properties}(d) $Q$ is snc-minimal, so $Q_1+Q_2$ is snc-minimal. By (i), $\kappa(Y'_C\setminus\{z\})\geq 0$. Suppose $(\ov S,Q_1+Q_2)$ is not almost minimal. We consider the proces of obtaining the almost minimal model $\pi_{\ell}\circ\ldots\circ \pi_1\:(\ov S,Q_1+Q_2)\to (\ov Z,T)$ as in the proof of \ref{lem:Y minimal}. The divisors $Q_1$ and $Q_2$ are the connected components of the boundary of $Y'_C\setminus\{z\}$, so because $Q_2$ cannot be contained in a divisor of quotient type, $T$ has at most one connected component of quotient type and we see that the Euler characteristic of the open part does not increase in this process. Since $\chi(\ov S\setminus(Q_1+Q_2))=0$, the log BMY inequality implies that $\chi(\ov Z\setminus T)\geq 0$, so in fact the Euler characteristic of the open part does not change in the process. Since $e(\ov Y_i,T_i)$ drops, it follows that $T$ is connected and $\ell=1$, i.e.\ the process has exactly one step. Because $T$ is not of quotient type, the log BMY inequality implies also that for $\cal P=(K_{\ov Z}+T)^+$ we have $\cal P^2=0$. In particular, $Y'_C\setminus\{z\}$ is not of log general type.

Let $L\not\subset Q_1+Q_2$ be the $(-1)$-curve on $\ov S$ witnessing the non almost-minimality of $(\ov S,Q_1+Q_2)$. Since $T$ is connected and not of quotient type, we have either $L\cdot Q_2=L\cdot Q_1=1$ or $L\cdot (Q+\ti C-Q_1)=0$, $L\cdot Q_1=1$ and $L+Q_1$ contracts to a smooth point. We have $\ks\cdot (\ks+Q_1+Q_2+L)=\ks\cdot(\ks+D+E)=2-\varepsilon$, so $K_{\ov Z}\cdot (K_{\ov Z}+T)\geq 2-\varepsilon$ and $(K_{\ov Z}+T)^2\geq -\varepsilon\geq -1$. We obtain $\Bk^2 T=(K_{\ov Z}+T)^2\geq -1$. But $T$ is a non-empty rational tree, so $\Bk T\neq 0$, hence $\Bk^2 T$ is a negative integer, i.e.\ $\Bk^2 T=-1$. Then $K_{\ov Z}\cdot (K_{\ov Z}+T)=1$ and $\varepsilon=1$. From Riemann-Roch $2K_{\ov Z}+T\geq 0$. We argue as in the proof of \ref{lem:Y minimal}, case $\gamma=4$, that $T$ is not a chain and $\cal P\cdot T=0$. Hence again by Fujita's classification $T$ is a fork with a branching component $H$. However, then the intersection matrix of $T-H$ is negative definite, so because $\#(T-H)=\rho(\ov Z)$, we get a contradiction with \ref{lem:Hodge lemma}.
\end{proof}

\blem\label{lem:delta-s-inequality} Let $(\ov X,T+R)$ be a smooth almost minimal pair, such that $T$ is a (connected) rational tree of non-quotient type and $R$ is a rational chain disjoint from $T$. Let $T_1,\dots, T_s$ be the maximal twigs of $T$. Assume $2K_{\ov X}+T+R\geq 0$ and put $\cal P=(K_{\ov X}+T+R)^+$ and $\delta(T)=\sum\limits_{i=1}^s\frac{1}{d(T_i)}$. Then $2\cal P^2+\delta(T)\geq s-2$.
\elem

\begin{proof} Put $\cal P=(K_{\ov X}+T+R)^+=K_{\ov X}+T+R-\Bk T-\Bk R$.  Since $\ovk(\ks +T+R)\geq 0$, $R$ and $T_i$'s are admissible, so are contained in the support of $\Bk (T+R)$. Since $\cal P$ is nef and $2K_{\ov X}+T+R$ is effective, we get $$\cal P\cdot T\leq \cal P\cdot (2K_{\ov X}+R+T)+\cal P\cdot T=2\cal P\cdot (K_{\ov X}+T+R)-\cal P\cdot R=2\cal P^2.$$ By the properties of barks (see the formula for the coefficients of components in $\Bk$ in \cite[??]{Miyan-OpenSurf}) and of the Zariski decomposition we have $\cal P\cdot R=0$ and $$\cal P\cdot T=(K_{\ov X}+T+R)\cdot T-\Bk T\cdot T=-2-\sum_{i=1}^s(\frac{1}{d(T_i)}-1)=s-2-\delta(T).$$
\end{proof}

\blem\label{lem:delta-eps-inequality} If $\varepsilon\leq 1$ or $\gamma\geq 5$ then the following inequalities hold: \benum[(i)]

\item $s-2-\frac{6}{c_h}\leq \delta(Q-Q_1+\ti C)\leq e(Q-Q_1+\ti C)\leq 1+\varepsilon -\frac{d'(Q_1)+d''(Q_1)-1}{c_h},$

\item $s-3-\varepsilon\leq-\frac{d'(Q_1)+d''(Q_1)-7}{c_h}$,
\eenum where $d'(Q_1)$, $d''(Q_1)$ denote determinants of the chain $Q_1$ with tips removed, $\delta$ and $s$ are defined as in \ref{lem:delta-s-inequality} with $T=Q+\ti C$. Analogous statements hold for the surface $Y_{\ti C}$.
\elem

\begin{proof} We have $\ks\cdot (\ks+Q+\ti C)=3-\varepsilon$ and $p_a(Q+\ti C)=-1$, so by Riemann-Roch $h^0(2\ks+Q+\ti C)\geq 2-\varepsilon$. So if $\varepsilon\leq 2$ then $2K_X+Q+\ti C\geq 0$. On the other hand, if $\gamma\geq 5$ and $\varepsilon\geq 2$ then $\varepsilon=t=2$ by \ref{prop:basic_inequality}, so by \ref{lem:k(Y)} $2K_X+Q+\ti C\geq 2K_X+Q\geq 0$. Applying \ref{lem:delta-s-inequality} to $(\ov X,T+R)=(\ov S,Q+\ti C)$ we have $s-2-2\cal P^2\leq \delta$. Since $\chi(Y_C\setminus\{z\})=0$, \ref{lem:BMY} gives $\cal P^2\leq \frac{3}{d(Q_1)}=\frac{3}{c_h}$. We have $-1-\varepsilon=(\ks+Q+\ti C)^2=\cal P^2+\Bk^2 (Q-Q_1+\ti C)+\Bk^2 Q_1=\cal P^2-e(Q-Q_1+\ti C)-e(Q_1)$ and, since $Q_1$ is an admissible chain, $e(Q_1)=\frac{1}{d(Q_1)}(d'(Q_1)+d''(Q_1)+2)$. Then $e(Q-Q_1+\ti C)\leq 1+\varepsilon-\frac{1}{c_h}(d'(Q_1)+d''(Q_1)-1)$. This gives (i) and hence (ii).
\end{proof}

\blem\label{lem:c1B-p1>=2} If $\ti j\geq 1$ then $c_1-\ti p_1\geq 2$. If $\ti j>1$ then $c_1-\ti c_1\geq 2$. \elem

\begin{proof}
After blowing twice on $\lambda$, the intersection of proper transforms of $\bar U$ and the line tangent to $\ti \lambda$ is $d-\ti c_1-\ti p_1=c_1-\ti p_1$ in case $\ti j=1$ and $d-\ti c_1-\ti c_1=c_1-\ti c_1$. Because $\bar U$ does not admit a good asymptote, the intersection is bigger than $1$.
\end{proof}

 \blem\label{lem:Q0_chain_for_gamma>5} If $\gamma\geq 6$ then $Q_0$ is a chain. \elem

\begin{proof} By \ref{lem:k(Y)} $\kappa(Y)\geq 0$. By \ref{prop:basic_inequality} $\varepsilon\leq 1$. It is now more convenient to treat $\lambda$ and $\ti \lambda$ in a symmetric way, so for the needs of this proof we temporarily cancel the assumption that $j\leq \ti j$. Suppose that $Q_0$ is not a chain.

\bcl $\{c_h,\ti c_{\ti h}\}=\{2,3\}$ or $\{2,4\}$. \ecl

\begin{proof} By \ref{lem:Q0_quotient_for_gamma>=6} $Q_0$ is of quotient type hence it is a contractible fork. The local fundamental group $\Gamma(Q_0)$ of the associated singular point is small and non-abelian, hence its order is at least $8$. We have $d(Q_1)=c_h$ and $d(\ti Q_1)=\ti c_{\ti h}$. By \ref{lem:BMY} $\frac{1}{c_h}+\frac{1}{\ti c_{\ti h}}\geq 1-\frac{1}{6}-\frac{1}{8}=\frac{17}{24}$, so $\min\{c_h,\ti c_{\ti h}\}=2$ and $\max\{c_h,\ti c_{\ti h}\}\leq 4$. Suppose $c_h=\ti c_{\ti h}=2$. Then equations \eqref{eq:1} and \eqref{eq:2} imply that $4$ divides $\gamma$, hence $\gamma=8$. Then $\varepsilon=0$ and $e(D+E)\leq 1$ by \ref{lem:eps_and_gamma}(v). Since $Q_1$ and $\ti Q_1$ are $(-2)$-tips of $D+E$, $D+E$ has no more tips, so $D$ is a chain. But then $Q_0$ is a chain; a contradiction. Thus $\{c_h,\ti c_{\ti h}\}=\{2,3\}$ or $\{2,4\}$.
\end{proof}

Because $Q_0$ is a fork, say, $\ti T_1$ is a branching component of $Q_0$ and $T_1$ is not. We write $Q_0=\ti T_1+R_1+R_2+R_3$, where $R_1$ is the twig of $Q_0$ containing the image of $L'_\infty$ and $R_2$ is the twig meeting $\ti C$. We have either $h=1$ or $h=2$ and $p_2=1$.

\bcl $\ti h\leq 3$. \ecl

\begin{proof} Suppose $\ti h\geq 4$. Then $R_2$ has at least three irreducible components, among them  a $\leq (-3)$-curve. Then $d(R_2)>5$, which implies that $Q_0$ is a fork of type $(2,2,n)$, hence $R_1=[2]$ and $R_3=[2]$. A determinant of a $(2,2,n)$-fork is $4(n(b-1)-d'(R_2))$, where $-b\leq -2$ is the self-intersection of the branching component and $d'(R_2)$ is the discriminant of $R_2$ with the last component removed. Since $Q_0$ is not a $(-2)$-fork, $d(Q_0)\geq 8$, hence $|\Gamma(Q_0)|\geq 16$ and the BMY inequality gives $\{c_h,\ti c_{\ti h}\}=\{2,3\}$. Since $R_1$ is irreducible, we have $h=1$ and $p_1=1$. In particular, $c_1\leq 3$. By \ref{lem:c1B-p1>=2} $\ti j=0$, because otherwise $c_1\geq \ti c_2+2\geq 4$. Then $\Psi=\id$, so since $R_1$ is irreducible, $j=0$, hence $R_1^2\leq -3$; a contradiction.
\end{proof}

\bcl $j,\ti j\geq 1$. \ecl

\begin{proof}
We have $j+\ti j =2+\varepsilon+\gamma-(h+\ti h)\geq 8-5=3$. Suppose that $\ti j=0$ or $j=0$. Then $\Psi=\id$ and since $j+\ti j\geq 3$, $R_1$ has at least four components. But $R_1$ contains a $\leq (-3)$-curve, so $d(R_1)>5$ and hence $Q_0$ is of type $(2,2,n)$. Then $R_2=[2]$, so $\ti h=2$ and $\ti p_2=2$. Then $\ti c_2\geq \ti p_2+1=3$, so $\ti c_2\in \{3,4\}$ and $c_h=2$. Because $\ti c_2$ and $\ti p_2$ are coprime, $\ti c_2=3$. We have $d(R_3)=\frac{\ti c_1}{\ti c_2}$, so since $R_3=[2]$, $\ti c_1=6$ and hence $\ti p_1=\ti c_2=3$. Since $c_h=2$ and $R_3=[2]$, $e(D+E)>1$, so $\varepsilon=1$ and hence $\gamma\leq 5+t=6$, i.e.\ $\gamma=6$. The total contribution of $E$, $Q_1$, $R_2$ and $R_3$ to $e(Q-\ti Q_1+C)$ is $\frac{1}{6}+\frac{1}{2}+\frac{1}{2}+\frac{1}{2}=\frac{5}{3}$, so since \ref{lem:delta-eps-inequality}(i) gives $e(Q-\ti Q_1+C)\leq \frac{5}{3}$, we have $h=1$. Then $(c_1,p_1)=(1,2)$, so \eqref{eq:1} gives $j+3\ti j=8$. But $j+\ti j=6$, so both $j$ and $\ti j$ are nonzero; a contradiction.
\end{proof}

\bcl $(h,\ti h)\neq (2,2)$. \ecl

\begin{proof}
Suppose $(h,\ti h)=(2,2)$. Then $j+\ti j=\varepsilon+\gamma-2$ and $t\leq 1$. By \ref{lem:j<=1} and by the previous claim $\min\{j,\ti j\}=1$.

Suppose $j=1$. Then $\ti j\geq 3+\varepsilon$. Let  $A$  be the component of $D'$ produced by the pair $\binom{c_1}{c_1} $. We have at least $3+\varepsilon$ successive contractions in $\Psi$ affecting $A$. Hence $A^2\leq -4-\varepsilon$, i.e.\ we blow up at least $3+\varepsilon$ times on $A$, otherwise $A$ becomes a nonnegative curve in $D$, which is impossible by \ref{lem:Psi_properties}(d). It follows that $R_1$ begins with $1+\varepsilon$ $(-2)$-curves and contains at least $2$ other components ($\Psi$ does not contract $T_1$ and the twig produced by $\binom{c_1}{p_1}$), hence its contribution to $e(D+E)$ is bigger than $\frac{1+\varepsilon}{2+\varepsilon}$. Since $p_2=1$, $Q_1$ consists of $(-2)$-curves, hence $e(D+E)>1$. It follows from \ref{lem:eps_and_gamma}(v) that $\varepsilon=1$. If $d(R_1)>5$ then $Q_0$ is of type $(2,2,n)$, so $R_2=[2]$ and hence $\ti p_2=2$. But $\{c_h,\ti c_{\ti h}\}=\{2,3\}$ or $\{2,4\}$, so in the latter case $\ti c_2=3$ and $c_2=2$ and we get $e(D+E)>\frac{2}{3}+\frac{1}{2}+\frac{1}{2}+\frac{1}{3}=2$, which contradicts \ref{lem:eps_and_gamma}(v). Thus $d(R_1)\leq 5$, hence $R_1=[2,2,2,2]$. Also, $Q_0$ is a fork of type $(2,3,5)$ or $(2,2,n)$, so $\{d(R_2), d(R_3)\}=\{2,3\}$ or $\{2,2\}$. It follows that $e(Q-\ti Q_1+C)\geq \frac{2}{3}+\frac{1}{2}+\frac{1}{d(R_2)}+\frac{1}{d(R_3)}+\frac{1}{\gamma}>2$; a contradiction by \ref{lem:delta-eps-inequality}(i).

Thus $\ti j=1$. Then $ j\geq 3+\varepsilon$. Let $\ti A$  be the component of $D'$ produced by the pair $\binom{\ti c_1}{\ti c_1} $. By the argument above $\ti A^2\leq -4-\varepsilon$, so we blow up at least $3+\varepsilon$ times on $\ti A$. In particular, $R_3$ begins with $1+\varepsilon$ $(-2)$-curves and it has at least $1$ component more. As above, we infer that $\varepsilon=1$, $\gamma=6$ and $j\geq 4$. Since $h>1$, $R_1$ has at least $2$ components ($\Psi$ does not contract $T_1$), so it is not a $(-2)$-curve. Since $Q_0$ is of quotient type, it follows that $d(R_3)\leq 5$, hence $R_3=[2,2,2]$ or $R_3=[2,2,2,2]$. Then $Q_0$ is a fork of type $(2,3,4)$ or $(2,3,5)$, so $d(R_2)\leq 3$. It follows that $e(Q-\ti Q_1+C)\geq \frac{1}{2}+\frac{1}{2}+\frac{1}{d(R_2)}+\frac{3}{4}+\frac{1}{\gamma}>2$; a contradiction by \ref{lem:delta-eps-inequality}(i).
\end{proof}

\bcl $h=1$. \ecl

\begin{proof}
Suppose $h\neq 1$. Then, by the previous claims $h=2$ and $\ti h=3$. Suppose that $\ti c_2>\ti p_2$. Then  $R_2$ has at least two components and a $\leq(-3)$-curve between them, so $d(R_2)\geq 5$. Since $\#R_1\geq 2$, $Q_0$ is a fork of type $(2,3,5)$, hence $R_2=[2,3]$, $R_1=[2,2]$ and $R_3=[2]$. It follows that $Q-\ti Q_1+C$ has at least four $(-2)$-tips, hence $e(Q-\ti Q_1+C)\geq 2$; a contradiction with \ref{lem:delta-eps-inequality}(i).

Therefore, $\ti c_2=\ti p_2$. Then $R_2$ contains a $\leq (-3)$-curve. Since $\#R_1>1$, $R_3=[2]$, so $\frac{\ti c_1}{\ti c_2}=2$. Also, denoting by $\ti A$ be the curve in $D'$ produced by the $\ti j$-th pair $\binom{\ti c_1}{\ti c_1}$ we have $\ti A^2=-3$.  Since $\ti c_2=\ti p_2$, $\ti T_1'$ is a $(-2)$-curve, so because $\ti T_1^2\leq -2$, $\Psi$ does not contract $\ti A$. It follows that $\ti j\geq 2$, otherwise $j=\gamma+\varepsilon-3-\ti j\geq 2$ and $\Psi$ contracts $\ti A$. Then $j=1$. Moreover, if $\Psi$ contracts the curve $A$ produced by the pair $\binom{c_1}{c_1}$ then it also contracts $\ti A$. Thus $A$ is not contracted by $\Psi$. We get $\#R_1\geq 3$. Since $R_2$ is not a $(-2)$-curve, $Q_0$ is of type $(2,3,4)$ or $(2,3,5)$. Then $R_1=[2,2,2]$ or $[2,2,2,2]$ and $d(R_2)=3$. The latter gives $R_2=[3]$, so $\ti c_2=2$. The divisor $Q-\ti Q_1+C$ has at least three $(-2)$-tips, so $e(Q-\ti Q_1+C)>\frac{3}{2}$. But by \ref{lem:delta-eps-inequality}(i) $e(Q-\ti Q_1+C)\leq \frac{3}{2}$; a contradiction.
\end{proof}

Since $h=1$, Claim 1 gives $c_1\leq 4$. Since $\ti j\geq 1$, by \ref{lem:c1B-p1>=2} we have $c_1\geq \ti p_1+2\geq \ti c_2+2\geq 4$, so Claim 1 gives $c_1=4$ and $\ti c_{\ti h}=2$. Then $\ti p_1=\ti c_2=2$ and by \ref{lem:c1B-p1>=2} $\ti j=1$, hence $j=\gamma+\varepsilon-\ti h\geq 6-\ti h+\varepsilon$. Because $R_2\neq 0$, we have $\ti h\geq 3$. Thus $\ti h=3$ and $(\ti c_2,\ti p_2,\ti c_3,\ti p_3)=(2,2,2,1)$. Put $\ti k=\ti c_1/\ti c_2=\ti c_1/2$. Clearly, $\ti k\geq 2$. We have $e(Q-\ti Q_1+C)=\frac{4-p_1}{4}+\frac{1}{\gamma}+(1-\frac{1}{\ti k})+\frac{1}{3}$. By \ref{lem:delta-eps-inequality}(i) $e(Q-\ti Q_1+C)\leq \varepsilon+\frac{1}{2}$, so $\varepsilon=1$ and $(\frac{4-p_1}{c_1}+\frac{1}{\gamma}\leq \frac{1}{\ti k}+\frac{1}{6})$. But since $\varepsilon=1$, \ref{prop:basic_inequality} gives $\gamma\leq 5+t\leq 6$, so $\gamma=6$. Then $p_1=3$ and $\ti k\leq 4$. Also, $j=4$. Then \eqref{eq:1} gives $\ti k=8$; a contradiction.
\end{proof}

Since by \ref{prop:gamma>1} $\gamma\geq 2$, the following proposition completes the proof of \ref{thm:main_result}(2).

\bprop\label{prop:gamma<=5} $\gamma\leq 5.$ \eprop

\begin{proof} As in the proof of \ref{lem:Q0_chain_for_gamma>5} we temporarily cancel the assumption that $j\leq \ti j$. Instead we may, and shall, assume $h\leq \ti h$. Suppose that $\gamma\geq 6$. By \ref{lem:Q0_chain_for_gamma>5} $Q_0$ is a chain, so $\ti h\leq 2$. By \ref{prop:basic_inequality} $\varepsilon\leq 1$.

Suppose that $h=2$. Then $\ti h=2$ and $t=2$, so $e(D+E)>1$, hence $\varepsilon=1$ and $\gamma\leq 7$. We have  $d'(Q_1)+d''(Q_1)\leq 7$ by \ref{lem:delta-eps-inequality}(ii). Hence $Q_1$, and similarly $\ti Q_1$, consist of at most three $(-2)$-curves. If, say, $Q_1=[2,2,2]$ then $\frac{d'(Q_1)+d''(Q_1)-1}{c_h}=\frac{5}{4}$ and we have a contradiction with \ref{lem:delta-eps-inequality}(i), because $e(Q-Q_1+\ti C)>(1-\frac{1}{\ti c_{\ti h}}) +\frac{1}{\gamma}\geq \frac{5}{6}$. So $Q_1$ and $\ti Q_1$ consist of at most two $(-2)$-curves. If $Q_1=[2]$ and $\ti Q_1=[2]$ then from \eqref{eq:1} and \eqref{eq:2} we get that $4$ divides $\gamma$, which is impossible. Similarly, it is not possible that $Q_1=[2,2]$ and $Q_2=[2,2]$, because otherwise $c_2=\ti c_2=3$ and then we get that $3$ divides $\gamma-2$. Thus, say, $Q_1=[2,2]$ and $\ti Q_1=[2]$. If one of the two tips of $Q_0$ is a $\geq (-3)$-curve then $e(Q-Q_1+\ti C)>\frac{1}{2}+\frac{1}{6}+\frac{1}{3}=1$, while $\frac{d'(Q_1)+d''(Q_1)-1}{d(Q_1)}=1$ and we have a contradiction with \ref{lem:delta-eps-inequality}(i).

Thus both tips of $Q_0$, call them $B$ and $\ti B$, are $\leq(-4)$-curves. Let $\alpha\: \ov S\to \ov N$ be the contraction of $C+Q_1$ and $\ti C+\ti Q_1$. Put $E_0=\alpha(E)$, $Z=\alpha(D)$. It does not touch $B+\ti B$. We have $E_0^2=-\gamma+5$, $K_{\ov N}\cdot(K_{\ov N}+Z)=\ks\cdot(\ks+D)+2=1-\ks\cdot E+2=5-\gamma$, $E_0\cdot(K_{\ov N}+Z)=E_0\cdot K_{\ov N}+5=\gamma-2$. Since $E_0\cdot Z>1$, the divisor $E_0+K_{\ov N}+Z$ is effective. Because $\kappa(K_{\ov N}+Z)=-\8$, we can find a maximal $m\in \mathbb{N}_+$ such that $|E_0+m(K_{\ov N}+Z)|\neq\emptyset$. Write $$E_0+m(K_{\ov N}+Z)=\sum A_i,$$ where $A_i$'s are irreducible. By the maximality of $m$ we have $|A_i+K_{\ov N}+Z|=\emptyset$, which gives $A_i\cong \PP^1$ and $A_i\cdot Z\leq 1$ for every $i$. We have $\rho(\ov N)>2$, because $\ov N$ contains negative curves $B$ and $\ti B$. Hence replacing successively $A_i$'s having non-negative self-intersection by singular members of their linear systems we may assume that $A_i^2<0$ for every $i$. We find $$(K_{\ov N}+Z)\cdot (E_0+B+\ti B+2K_{\ov N})=E_0\cdot(K_{\ov N}+Z)-2+2K_{\ov N}\cdot(K_{\ov N}+Z)=6-\gamma\leq 0$$ and $E_0\cdot(E_0+B+\ti B+2K_{\ov N})=\gamma-9<0$. Therefore, there exists $A_{i_0}$ for which $A_{i_0}\cdot (E_0+B+\ti B+2K_{\ov N})<0$. The intersection of $B$ (respectively $\ti B$) with $E_0+B+\ti B+2K_{\ov N}$ equals $-2+B\cdot K_{\ov N}\geq 0$ (respectively $-2+B\cdot K_{\ov N}\geq 0$), so $A_{i_0}\neq B, \ti B$. Also, $A_{i_0}\neq E_0$, because $|E_0+K_{\ov N}+Z|\neq \emptyset$. Thus $A_{i_0}$ is a $(-1)$-curve and $A_{i_0}\cdot E_0<2$. Since by the definition of $\alpha$ the curve $E_0$ meets every $(-1)$-curve in $Z$ at least twice, $A_{i_0}$ is not a component of $Z$. Then $A_{i_0}\setminus Z\subset S$ is a good asymptote of $U$; a contradiction.

Thus $h=1$. Suppose $\ti h=2$. Then $d(Q_1)=c_1$, $d(\ti Q_1)=\ti c_2$ and $\ti Q_1$ is a $(-2)$-chain. Now \ref{lem:BMY} gives $\frac{1}{c_1}+\frac{1}{\ti c_2}\geq \frac{5}{6}-\frac{1}{d(Q_0)}$. By \ref{lem:j<=1} $\min\{j,\ti j\}\leq 1$.

Suppose that $j=1$. Let $A$ be the curve in $D'$ produced by the pair $\binom{c_1}{c_1}$. By \eqref{eq:h+hB} $\ti j=\gamma+\varepsilon-2\geq 4+\varepsilon$, so $L_\8'$ is a $(-1)$-curve and we have at least $4+\varepsilon$ contractions in $\Psi$ affecting $A$. By \ref{lem:Psi_properties}(d) we get $A^2\leq -5-\varepsilon$. It follows that we blow up at least $4+\varepsilon$ times on the proper transform of $A$ in the pair $\binom{c_1}{p_1}$, so $Q_1$ has at least $3+\varepsilon$ components. But \ref{lem:delta-eps-inequality}(ii) gives $1+\frac{d'(Q_1)+d''(Q_1)-7}{d(Q_1)}\leq \varepsilon$. For $\varepsilon=0$ we get $d(Q_1)+d'(Q_1)+d''(Q_1)\leq 7$, which is impossible, because $Q_1$ has at least three components. For $\varepsilon=1$ we get $d'(Q_1)+d''(Q_1)\leq 7$, which is also impossible, because now $Q_1$ has at least four components.

Suppose $\ti j=1$. Let $\ti A$ be the curve in $D'$ produced by the pair $\binom{\ti c_1}{\ti c_1}$. We have $j=\gamma+\varepsilon-2\geq 4+\varepsilon$, so $L_\8'$ is a $(-1)$-curve and we have at least $4+\varepsilon$ contractions in $\Psi$ affecting $\ti A$. By \ref{lem:Psi_properties}(d) we get $\ti A^2\leq -5-\varepsilon$. It follows that we blow up at least $4+\varepsilon$ times on the proper transform of $\ti A$ in the pair $\binom{\ti c_1}{\ti p_1}$, so $\ti R$, the twig of $D+E$ produced by the pair $\binom{\ti c_1}{\ti p_1}$ has at least $3+\varepsilon$ components and begins with at least $2+\varepsilon$ $(-2)$-curves. Therefore, its contribution to $e(D+E)$ (and to $e(Q-Q_1+\ti C)$), is bigger than $\frac{2+\varepsilon}{3+\varepsilon}$. Since $\ti Q_1$ is a $(-2)$-chain we get $e(D+E)>1$, so $\varepsilon=1$ by \ref{lem:eps_and_gamma}(v). By \ref{prop:basic_inequality} $\gamma\leq 5+t\leq 7$. We have $\#Q_0\geq \#\ti R+2\geq 6$, hence $d(Q_0)\geq 7$. Now \ref{lem:BMY} gives $\frac{1}{c_1}+\frac{1}{\ti c_2}\geq 1-\frac{1}{\gamma}-\frac{1}{d(Q_0)}>\frac{2}{3}$, so $\min\{c_1,\ti c_2\}=2$ and $\max\{c_1,\ti c_2\}\leq 5$. But since $\ti j=1$, \ref{lem:c1B-p1>=2} gives $c_1\geq \ti p_1+2\geq \ti c_2+2\geq 4$, so $\ti c_2=2$ and $c_1\in\{4,5\}$. Denoting by $e_1$ the contribution of the twig of $Q_0$ meeting $C$ to $e(Q-Q_1+\ti C)$ we get from \ref{lem:delta-eps-inequality}(i) that $e_1+\frac{d'(Q_1)+d''(Q_1)-1}{c_1}<2-\frac{1}{2}-\frac{3}{4}-\frac{1}{7}=\frac{17}{28}<\frac{2}{3}$. Since $c_1=4,5$, the latter inequality implies that $Q_1$ is irreducible. Then the tip of $Q_0$ meeting $C$ is a $(-2)$-curve. Indeed, it is not contracted by $\Psi$, its proper transform on $\ov S'$ is a $(-2)$-curve and $Q_0$ contains no curves of non-negative self-intersection, hence the tip is not touched by $\Psi$. We obtain $e_1\geq\frac{1}{2}$, hence $\frac{1}{5}\leq\frac{1}{c_1}<\frac{2}{3}-\frac{1}{2}=\frac{1}{6}$; a contradiction.

Thus, we have $\min\{j,\ti j\}=0$. Suppose $j=0$. Then $\Psi=\id$ and $\ti j=\varepsilon +\gamma-1\geq 5$. It follows that $\# Q_0\geq 8$. The component produced by the last pair $\binom{\ti c_1}{\ti c_1}$ is a $(\leq -3)$-curve. Also, $(\ti T_1)^2\leq -3$. We find $d(Q_0)\geq 53$. From \ref{lem:BMY} we get that $c_1\leq 3$. But $c_1\geq \ti p_1+2\geq \ti c_2+2\geq 4$; a contradiction.

Now suppose that $\ti j=0$. We find that $\#Q_0\geq 8$ and (because $\ti T_1$ and the tip contained in $F$ are $\leq (-3)$-curves) that $d(Q_0)\geq 43$. If $c_1=\ti c_2=2$ then $e(D+E)>1$, so $\varepsilon\geq 1$ and hence $\gamma\leq 7$. But in the latter case \eqref{eq:1} and \eqref{eq:2} give $4|\gamma$. Thus $\max\{c_1,\ti c_2\}\geq 3$. Now \ref{lem:BMY} gives $\{c_1,\ti c_2\}=\{2,3\}$ and $\gamma=6$. Then $j=5+\varepsilon$. Write $\ti c_1=\ti k\ti c_2$ and $\ti p_1=\ti l\ti c_2$. Multiplying \eqref{eq:1} by $c_1$ and subtracting  \eqref{eq:2} we get $$\ti c_2^2\ti k(\ti k-\ti l)+c_1\ti c_2\ti l=6(c_1-1)+c_1^2-c_1+\ti c_2.$$ If $(c_1,\ti c_2)=(3,2)$ we get $2\ti k(\ti k-\ti l)+3\ti l=10$, so $\ti l$ is even and hence $10\geq 2(\ti l+1)+3\ti l\geq 6+6$; a contradiction. Thus $(c_1,\ti c_2)=(2,3)$. We get $3(3\ti k(\ti k-\ti l)+2\ti l)=11$; a contradiction.

We are left with the case $h=\ti h=1$. We have now $j+\ti j=\varepsilon+\gamma\geq 6$. Suppose that $d(Q_0)>6$. Then \ref{lem:BMY} gives $\frac{1}{c_1}+\frac{1}{\ti c_1}\geq 1-\frac{1}{\gamma}-\frac{1}{d(Q_0)}>\frac{2}{3},$ so $\min\{c_1, \ti c_1\}=2$. By symmetry, we may assume that $\ti c_1=2$. Then $j=0$, because otherwise $\ti c_1\geq p_1+2$. It follows that $\#Q_0\geq 7$. Since the curve produced by the last pair $\binom{\ti c_1}{\ti c_1}$ is a $\leq (-3)$-curve, $d(Q_0)\geq 15$. Then $\frac{1}{c_1}\geq 1-\frac{1}{6}-\frac{1}{15}-\frac{1}{2}=\frac{4}{15}$, so $c_1\leq 3$. But by \ref{lem:c1B-p1>=2} we have $c_1\geq \ti c_1+2$, because $\ti j>1$; a contradiction.

Thus $d(Q_0)\leq 6.$ Then $j>0$ and $\ti j>0$, because otherwise $\#Q_0\geq j+\ti j+1\geq 7$ and $d(Q_0)\geq 8$. We may assume $c_1\geq \ti c_1$. Then \ref{lem:c1B-p1>=2} implies that $j=1$. We obtain $\ti j=\gamma+\varepsilon-1\geq 5$. Again by \ref{lem:c1B-p1>=2} $\ti c_1\geq p_1+2\geq 3$ and $c_1\geq \ti c_1+2$, so $\ti c_1\geq 3$ and $c_1\geq 5$. Then \ref{lem:BMY} gives $\ti c_1=3$ and $d(Q_0)\leq 3$. In particular, $p_1=1$. Since $C'+\ti C'$ is not touched by $\Psi$, we have $\#Q_0\geq 2$, hence $Q_0=[2,2]$. Then $\frac{1}{c_1}\geq \frac{1}{3}-\frac{1}{\gamma}$, so $c_1\leq 6$. The formulas \eqref{eq:1} and \eqref{eq:2} read as
\begin{align*}c_1+\gamma+5&=3\ti j+\ti p_1,\\
5c_1+\gamma+9&=9\ti j+3\ti p_1.
\end{align*}
We obtain $3(c_1+\gamma+5)=5c_1+\gamma+9$, so $c_1=\gamma+3\geq 9$; a contradiction.
\end{proof}

\section{\texorpdfstring{Type $(0,\ti j)$}{Type (0,j)}}

In the following three subsections we prove the following lemma.

\blem\label{prop:type_(0,0-1-2)} If $U$ is of type $(0,\ti j)$ for some $\ti j\leq 2$ then we can make a change of coordinates on $S$, after which $U$ still has separate non-simple branches at infinity, it is of type $(0,\ti m)$ for some $\ti m\in \Z$ and the degree of $U$ drops.\elem

\subsection{\texorpdfstring{Type $(0,0)$}{Type (0,0)}}Here we prove \ref{prop:type_(0,0-1-2)} in case $j=\ti j=0$.

\begin{proof}[Proof of \ref{prop:type_(0,0-1-2)}] Let $c_1=kc_2$ and $\ti c_1=\ti k\ti c_2$. Let $c_1-p_1=\beta  c_2$, $\ti c_1-\ti p_1=\ti \beta  \ti c_2$. Then $\beta, \ti \beta \geq 1$.  We have $\gamma=\gamma'$ by \ref{lem:Psi_properties}(b), so the formulas \eqref{eq:1} and \eqref{eq:2} take the form
\begin{align}
  \gamma+(\beta +k)c_2+(\ti \beta +\ti k)\ti c_2&=p_2+\dots+p_h+\ti p_2+\dots+\ti p_h,\label{eq:1_case00}\\
  \gamma+\beta kc_2^2+\ti \beta \ti k\ti c_2^2+2k\ti kc_2\ti c_2&=p_2c_2+\dots+p_hc_h+\ti p_2\ti c_2+\dots+\ti p_h\ti c_h.\label{eq:2_case00}
\end{align}

We may assume that $c_2\geq \ti c_2$. We have $\gamma\leq 5$, so \ref{prop:basic_inequality} gives $h+\ti h=\gamma+\varepsilon+2\leq 9$. We have $2k\ti kc_2\ti c_2\geq 8\ti c_2^2\geq\sum\limits_{i\geq 2}\ti p_i\ti c_i$. It follows from \eqref{eq:2_case00} that \begin{equation} \beta kc_2^2<\sum\limits_{i\geq 2}p_i c_i \label{eq:3_case00}.\end{equation} We obtain $\beta kc_2^2<(h-1)c_2^2$ so $\beta k\leq h-2\leq 6$. Let $r$ be the number of pairs $\binom{c_2}{c_2}$, i.e.\ $r$ is such that $p_{r+2}$ is the first $p_i$ smaller than $c_2$. We have $$\sum\limits_{i\geq 2} p_i c_i\leq rc_2^2+c_2p_{r+2}+\sum_{i\geq r+3}c_ip_i\leq rc_2^2+c_2(c_2-c_{r+3})+\sum_{i\geq r+3}c_ip_i\leq (r+1)c_2^2-2c_{r+3}^2+(h-r-2)c_{r+3}^2,$$ hence \begin{equation}
\sum\limits_{i\geq 2} p_i c_i\leq  (r+1)c_2^2+(h-r-4)c_{r+3}^2\label{eq:3p_case00}.
\end{equation}

Suppose $\beta=2$. Since $\beta k\leq h-2\leq 6$, we get $\beta=2$, $k=3$, $h=8$ and then consequently $\ti h=1$, $\gamma=5$, $\varepsilon=2$ and by \ref{prop:basic_inequality} $t=2$, hence $\ti p_1=p_8=1$. The equations \eqref{eq:3_case00} and \eqref{eq:3p_case00} give $6<r+1+\frac{4-r}{4}$, so $r\geq 6$. Because $h=8$ we get $r=6$ and $(c_h,p_h)=(c_2,1)$. Then \eqref{eq:2_case00} reads as $5+\ti c_1+c_2(6\ti c_1-1)=0$; a contradiction. If $r\leq k-2$ then \eqref{eq:3p_case00} and \eqref{eq:3_case00} give $4c_{r+3}^2\leq c_2^2<(h-r-4)c_{r+3}^2$ and hence $r+9\leq h\leq 9-\ti h\leq 8$, which is impossible.

Thus $\beta=1$ and $r\geq k-1$. Suppose $\ti c_1\geq c_2$. Using \eqref{eq:2_case00} we obtain that  $$7c_2^2\geq(h+\ti h-2)c_2^2\geq \sum\limits_{i\geq 2} p_i c_i+\sum\limits_{i\geq 2}^{\ti h} \ti p_i \ti c_i>kc_2^2+2kc_2\ti c_1\geq kc_2^2+2k c_2^2=3kc_2^2,$$ hence $k=2$. We have also $\gamma+\epsilon=h+\ti h-2=7$. By \ref{prop:basic_inequality} and \ref{prop:gamma<=5} $\gamma=5$, $\varepsilon=2$ and $t=2$. From \eqref{eq:3_case00} it follows that $h\geq 2$. If $\ti h\geq 2$ then $$3kc_2^2=6c_2^2<\sum\limits_{i\geq 2} p_i c_i+\sum\limits_{i\geq 2}^{\ti h} \ti p_i \ti c_i\leq 5c_2^2+c_h+\ti c_h\leq 5c_2^2+2c_2,$$ which is impossible, because $c_2\geq 2$. Therefore, $h=8$ and $\ti h=1$. Now $\sum\limits_{i\geq 2} p_i c_i\leq 6c_2^2+c_2$, which together with \eqref{eq:2_case00} gives $2c_2^2+4c_2\ti c_1<6c_2^2+c_2$. From this we get $c_2\geq \ti c_1$. Hence $c_2=\ti c_1$. Because $h=8$ and $\ti h=1$, the formulas \eqref{eq:1_case00} and \eqref{eq:1_case00} give $\gamma-\ti p_1\equiv p_8 \mod c_8$ and $\gamma\equiv 0\mod c_8$, so $c_8$ divides $\gamma=5$ and $p_8+\ti p_1$. But $t=2$, so $p_8+\ti p_1=2$; a contradiction.

Therefore, $\ti c_1<c_2$. Recall that $r-k+1\geq 0$. After blowing up on $\lambda$ according to $\binom{c_1}{p_1}\binom{c_2}{c_2}_{k-1}$ and then successively contracting $(-1)$-curves starting from $L_\infty'$ we get new coordinates on $S$ in which the type of $U$ at infinity is $(0,r-k+1)$, the branches at infinity are separated and not simple ($c_2>1$). The degree of $U$ in these new coordinates is $k\ti c_1+c_2$ and, since $\ti c_1<c_2$, it is smaller than the original degree $c_1+\ti c_1$. Indeed, $k\ti c_1+c_2<kc_2+\ti c_1\iff (k-1)\ti c_1<(k-1)c_2$.
\end{proof}

\subsection{\texorpdfstring{Type $(0,1)$}{Type (0,1)}}We prove Proposition \ref{prop:type_(0,0-1-2)} for $(j,\ti j)=(0,1)$. Put $P=\sum\limits_{i\geq 2}p_i$ and $\ti P=\sum\limits_{i\geq 2}\ti p_i$.

\begin{proof}[Proof of \ref{prop:type_(0,0-1-2)}] The formulas \eqref{eq:1} and \eqref{eq:2} read as

\begin{align}
  \gamma+2c_1+\ti c_1&=p_1+\ti p_1+P+\ti P.\label{eq:1_case01}\\
  \gamma+c_1^2+2c_1\ti c_1&=p_1c_1+\ti p_1\ti c_1+\sum\limits_{i\geq 2} p_i c_i+\sum\limits_{i\geq 2}\ti p_i  \ti c_i.\label{eq:2_case01}
\end{align}
Write $c_1=\ti p_1+\theta$. By \ref{lem:c1B-p1>=2} $\theta\geq 2$. We keep the notation from the previous section. We rewrite the formulas:
\begin{align}
  \gamma+\beta c_2+\theta+\ti c_1&=P+\ti P,\label{eq:3_case01}\\
  \gamma+\beta c_2c_1+c_1\ti c_1+\theta \ti c_1&= \sum\limits_{i\geq 2} p_i c_i+\sum\limits_{i\geq 2} \ti p_i \ti c_i.\label{eq:4_case01}
\end{align}
Suppose that $\ti c_2\geq c_2$. Then $\ti c_2(P+\ti P)\geq \sum\limits_{i\geq 2} p_i c_i+\sum\limits_{i\geq 2} \ti p_i \ti c_i$, so multiplying \eqref{eq:3_case01} by $\ti c_2$ and subtracting \eqref{eq:4_case01} we get $$\gamma\ti c_2+\beta c_2\ti c_2+\theta\ti c_2+\ti c_1\ti c_2 \geq \gamma+\beta c_2c_1+c_1\ti c_1+\theta\ti c_1,$$ and then subsequently
\begin{align*}
  \gamma\ti c_2+\beta c_2\ti c_2 &>\beta c_2c_1+\ti c_1(c_1-\ti c_2)+\theta(\ti c_1-\ti c_2),\\
  \gamma\ti c_2+\beta c_2\ti c_2 &>\beta c_2c_1+\ti c_1\theta+\theta(\ti k-1)\ti c_2,\\
  \beta c_2(\ti c_2-c_1) &>\ti c_2(\ti k\theta-\gamma+\theta(\ti k-1)).
\end{align*}
Because $c_1>\ti p_1\geq \ti c_2$, we infer $\ti k\theta -\gamma+\theta(\ti k-1)<0$. It follows that $\theta(2\ti k-1)<\gamma\leq 5$, so $\ti k<2$; a contradiction.

Thus $c_2>\ti c_2$. We now show that \begin{equation}\beta c_2 c_1=k\beta c_2^2<\sum\limits_{i\geq 2} p_i c_i.\label{eq:5_case01} \end{equation} Suppose \eqref{eq:5_case01} fails. By \eqref{eq:4_case01} $c_1\ti c_1+\theta\ti c_1<\sum\limits_{i\geq 2}\ti p_i\ti c_i\leq \ti P\ti c_2$, so by \eqref{eq:3_case01} $\ti k c_1+\ti k\theta\leq \ti P-1\leq \gamma-1+\beta c_2+\ti k(c_2-1)+\theta$. Then $0\leq 2+3-\gamma\leq \theta(\ti k-1)+\ti k+1-\gamma\leq c_2(\beta+\ti k-\ti k k)$, so $\ti k(k-1)\leq \beta\leq k-1$. But $\ti k\geq 2$; a contradiction.

We have $h+\ti h=\gamma+\varepsilon+1\leq 8$, so $(h-1)+(\ti h-1)\leq 6$ and \eqref{eq:5_case01} gives now $(\beta+1)\beta c_2^2<\sum\limits_{i\geq 2} p_i c_i\leq 6c_2^2$, so $\beta =1$.

Suppose $\ti c_1\geq c_2$. We multiply \eqref{eq:3_case01} by $c_2$. Since $c_2>\ti c_2$ we obtain that $\gamma c_2+c_2^2+\theta c_2+c_2\ti c_1\geq \gamma +c_2c_1+c_1\ti c_1+\theta\ti c_1.$   We get $$\gamma c_2> c_2^2(k-1)+\theta(\ti c_1-c_2)+\ti c_1c_2(k-1)\geq c_2^2(k-1)+\ti c_1c_2(k-1).$$ Because $c_2\geq \ti c_2+1\geq 2$, the above inequality gives $$4\geq \gamma-1\geq (k-1)(c_2+\ti c_1)\geq 4(k-1),$$ hence $k=2$, $c_2=\ti c_1=2$ and $c_1=4.$ Then $\ti p_1=1$, which gives $\theta=c_1-\ti p_1=3$. We have also $\ti h=1$. Now \eqref{eq:3_case01} and \eqref{eq:4_case01} give $\gamma+7=P=2(h-2)+1$ and $\gamma+22=4(h-2)+2$, hence $\gamma=8, h=9$; a contradiction since $h+\ti h\leq 8$.

Thus $\ti c_1<c_2$. The inequality \eqref{eq:3p_case00} holds, so we prove as in the previous section that $r\geq k-1$ and then that we may change coordinates on $S$ so that in the new coordinates $U$ the branches at infinity are non-simple and separated, the type at infinity of $U$ is $(0,r-k+1)$ and its degree equals $k\ti c_1+c_2$. Since $\ti c_1<c_2$, we get $k\ti c_1+c_2<kc_2+\ti c_1=c_1+\ti c_1$, so the degree drops.
\end{proof}

\subsection{\texorpdfstring{Type $(0,2)$}{Type (0,2)}}Here we prove Proposition \ref{prop:type_(0,0-1-2)} in the remaining case $(j,\ti j)=(0,2)$.

\begin{proof}[Proof of \ref{prop:type_(0,0-1-2)}] The formulas \eqref{eq:1} and \eqref{eq:2} read as
\begin{align}
  \gamma+2c_1&=p_1+\ti p_1+P+\ti P.\label{eq:1_case02}\\
  \gamma+c_1^2+2c_1\ti c_1&=\ti c_1^2+c_1p_1+\ti c_1\ti p_1+\sum\limits_{i\geq 2} p_i c_i+\sum\limits_{i\geq 2} \ti p_i \ti c_i.\label{eq:2_case02}
\end{align}

Put $a=c_1-\ti c_1$. By \ref{lem:c1B-p1>=2} $a\geq 2$. In particular, $c_1>\ti c_1$. We rewrite the formulas as:
\begin{align}
  \gamma+\beta c_2+c_1-\ti p_1&=P+\ti P.\label{eq:3_case02}\\
  \gamma+\beta c_1c_2+c_1\ti c_1+\ti c_1(c_1-\ti p_1-\ti c_1)&=\sum\limits_{i\geq 2} p_i c_i+\sum\limits_{i\geq 2} \ti p_i \ti c_i.\label{eq:4_case02}
\end{align}

Suppose $c_1-\ti p_1-\ti c_1<0$.  Multiplying \eqref{eq:1_case02} by $c_1$ and subtracting \eqref{eq:2_case02} we get $$\gamma(c_1-1)+c_1^2-2c_1\ti c_1=\ti p_1(c_1-\ti c_1)-\ti c_1^2+\sum\limits_{i\geq 2}p_i(c_1-c_i)+\sum\limits_{i\geq 2}\ti p_i(c_1-\ti c_i),$$ hence \begin{equation}\gamma(c_1-1)+(c_1-\ti c_1)(c_1-\ti c_1-\ti p_1)=\sum\limits_{i\geq 2}p_i(c_1-c_i)+\sum\limits_{i\geq 2}\ti p_i(c_1-\ti c_i).\label{eq:6_case02}\end{equation} Because $\frac{1}{2}c_1>\frac{1}{2}\ti c_1\geq \ti c_2$, we have $c_1-\ti c_i\geq \frac{1}{2}c_1$ for $i\geq 2$, so  \eqref{eq:6_case02} gives $\gamma>\frac{1}{2}(P+\ti P)$. From \eqref{eq:3_case02} we infer that $$\beta c_2+c_1-\ti p_1=P+\ti P-\gamma\leq\gamma-1\leq 4.$$  Since $c_1-\ti p_1\geq c_1-\ti c_1+1\geq 3$, we obtain that $\beta=c_2=1$, so $h=1$. Also $\gamma=5$, $c_1-\ti p_1=3$ and $\ti c_1-\ti p_1=1$. Thus $\ti c_2=1$ and $\ti h=1$. Then $\varepsilon+\gamma=h+\ti h=2$, so $\varepsilon=0$ and $\gamma=2$. We get $\ks\cdot(\ks+D)=2-\varepsilon-\ks\cdot E=2$. The Riemann-Roch theorem gives $-\ks-D\geq 0$. Because $2\ks+D+E\geq 0$ by \ref{cor:eps=3-h0}, we have $\ks+E\geq 0$. This implies that $\ks\geq 0$; a contradiction.

Thus $c_1-\ti p_1-\ti c_1\geq 0$. If $\ti c_2\geq c_2$ then multiplying \eqref{eq:3_case02} by $\ti c_2$ and subtracting \eqref{eq:4_case02} we get \begin{equation}\gamma\ti c_2+\beta c_2\ti c_2+\ti c_2(c_1-\ti p_1)>\beta c_1c_2+c_1\ti c_1+\ti c_1(c_1-\ti p_1-\ti c_1)> (\beta c_2+c_1) \ti c_1\label{eq:5_case02},\end{equation} hence $\gamma+\beta c_2+c_1-\ti p_1>(\beta c_2+c_1)\ti k$. But then $\gamma>(\beta c_2+c_1)(\ti k-1)+\ti p_1\geq c_1+1\geq \ti c_1+3\geq 5$, which is impossible. Thus  $c_2>\ti c_2$.  By \eqref{eq:4_case02} we have $$\beta k c_2^2<\gamma+\beta c_1c_2+c_1\ti c_1\leq \sum\limits_{i\geq 2}p_i c_i+\sum\limits_{i\geq 2} \ti p_i \ti c_i\leq (h+\ti h-2)c_2^2=(\gamma+\varepsilon-2)c_2^2\leq 5c_2^2,$$ so $\beta=1$. We write the formulas in the following form
\begin{align*}\gamma +c_2+\ti c_1 +(c_1-\ti p_1-\ti c_1) &=P+\ti P \\
\gamma+kc_2^2+kc_2\ti c_1+\ti c_1(c_1-\ti p_1-\ti c_1)&=\sum\limits_{i\geq 2} p_ic_i+\sum\limits_{i\geq 2} \ti p_i\ti c_i.
\end{align*}
Multiplying the first one by $c_2$ and subtracting the second one we get $$\gamma(c_2-1)\geq c_2^2(k-1)+\ti c_1c_2(k-1)+(c_1-\ti p_1-\ti c_1)(\ti c_1-c_2).$$ In case $\ti c_1\geq c_2$ we get $$4\geq \gamma-1\geq (k-1)(c_2 +\ti c_1)\geq (k-1)(\ti c_2+1+\ti c_1)\geq 4(k-1),$$ so $k=2$, $c_2=2$, $\ti c_1=2$ and $\gamma=5$. But in the latter case the previous inequality gives $5\geq 4+4$.

Thus $\ti c_1<c_2$. It follows that $\ti c_2<\frac{1}{2}c_2$. The inequality \eqref{eq:3p_case00} holds, so by \eqref{eq:4_case02}
\begin{equation*}
kc_2^2<\sum \limits_{i\geq 2} p_ic_i+\sum\limits_{i\geq 2}\ti p_i\ti c_i\leq \sum\limits_{i\geq 2} p_i c_i+(\ti h-1)\ti c_2^2\leq  (r+1)c_2^2+(h-r-4)c_{r+3}^2 +(\ti h-1)\frac{c_2^2}{4}.
\end{equation*}
Note that $c_2\geq \ti c_1+1>1$, so $h\geq 2$ and hence $\ti h=\gamma+\varepsilon-h\leq 7-h\leq 5$.

We have $r\geq k-1$, otherwise the above inequality gives $$0\leq (1-\frac{\ti h-1}{4})c_2^2<(h-r-4)c_{r+3}^2\leq \frac{c_2^2}{4}(h-r-4),$$ hence $9+r<h+\ti h$, which is impossible, because $h+\ti h\leq 7$. As in the previous two subsections, we can now change coordinates on $S$ so that $U$ has non-simple, separated branches at infinity, it is of type $(0,\ti m)$ for some $\ti m\geq 0$ and its degree is $k\ti c_1+c_2$. Since $\ti c_1<c_2$, we have $k\ti c_1+c_2<kc_2+\ti c_1=c_1+\ti c_1$, so the degree drops.
\end{proof}

\subsection{\texorpdfstring{Types $(0,\ti j)$ for $\ti j\geq 3$}{Types (0,j) for j>=3}} In this section we show the following proposition. Recall that we assume the coordinates on $S=\C^2$ are chosen so that branches of $U$ at infinity are disjoint.

\bprop\label{prop:j>0} $U$ is of type $(1,\ti j)$ for some $\ti j\geq 1$. \eprop

\begin{proof}

Note that by \ref{lem:j<=1} $j\leq 1$. Suppose $j=0$. By \ref{prop:type_(0,0-1-2)} and by induction on the degree of $\bar U$ we may assume that $\ti j\geq 3$. We are going to show this is impossible. Let $A$ be a conic in $\PP^2$ which meets $\lambda$ and which follows $\ti \lambda$ during the first three blowing ups. Blow up once over $\lambda$ and four times over $\ti \lambda$. If $\ti j=3$ then the intersection of the proper transforms of $A$ and $\bar U$ at this stage is $2d-p_1-3\ti c_1-\ti p_1=2c_1-p_1-\ti c_1-\ti p_1$ and for $\ti j>3$ it is $2d-p_1-4\ti c_1=2c_1-p_1-2\ti c_1$. It follows that

\begin{equation}\label{eq:noasy_case03}
  \begin{aligned}
  2c_1-\ti c_1&\geq \ti p_1+p_1 &\text{if } \ti j=3,\\
  2(c_1-\ti c_1)& \geq p_1 &\text{\ if } \ti j>3.
  \end{aligned}
\end{equation}

Put $a=c_1-\ti c_1$. Again, by \ref{lem:c1B-p1>=2} $a\geq 2$. We keep the notation $P=\sum\limits_{i\geq 2}p_i$ and $\ti P=\sum\limits_{i\geq 2}\ti p_i$. The formulas \eqref{eq:1} and \eqref{eq:2} read as
\begin{align}
\gamma+2c_1 &=p_1+(\ti j-2)\ti c_1+\ti p_1+P+\ti P.\label{eq:1_case03}\\
\gamma+c_1^2+2c_1\ti c_1 &=c_1p_1+(\ti j-1)\ti c_1^2+\ti c_1\ti p_1+\sum\limits_{i\geq 2}p_ic_i+\sum\limits_{i\geq 2}\ti p_i\ti c_i.\label{eq:2_case03}
\end{align}
Put $x=2c_1-p_1-(\ti j-2)\ti c_1-\ti p_1=\beta+a-(\ti j-3)\ti c_1-\ti p_1$. We rewrite the equations in the following form
\begin{align}
  \gamma+x &=P+\ti P.\label{eq:3_case03}\\
  \gamma+a(c_1-p_1+\ti c_1)+x\ti c_1 &= \sum\limits_{i\geq 2}p_ic_i+\sum \limits_{i\geq 2}\ti p_i\ti c_i.\label{eq:4_case03}
\end{align}
We divide the proof into two cases.

\bca $x\geq 0$. \eca

Suppose that $\ti c_2\geq c_2$. Then multiplying \eqref{eq:3_case03} by $\ti c_2$ and subtracting \eqref{eq:4_case03} we get
$$\gamma(\ti c_2-1)\geq a(c_1-p_1+\ti c_1)+(\ti c_1-\ti c_2)x,$$ so $\gamma-1\geq a\ti k+x$. From this we obtain that $c_1-\ti c_1=a=2$, $\ti c_1=2\ti c_2$, $\gamma=5$, and $2c_1-\ti c_1-p_1-\ti p_1=x=0$.
From \eqref{eq:3_case03} we get $P+\ti P= 5$, which is possible only if $(h-1)+(\ti h-1)\leq 3$. From \eqref{eq:4_case03} we obtain that $5+2\beta c_2+4\ti c_2\leq Pc_2+\ti P\ti c_2$, hence $(2\beta+\ti P-5)c_2<(\ti P-4)\ti c_2$. If $\ti P-4\leq 0$ then  $(\ti P-4)\ti c_2\leq (\ti P-4)c_2$, hence $2\beta+\ti P-5<\ti P-4$, so $2\beta<1$, which is impossible. Thus $\ti P=5$ and hence $P=0$. Because $x=0$ and $\ti j\geq 3$, we have $\beta c_2\geq \beta\geq \ti p_1-a=\ti p_1-2$, so the inequality gives $1+2\ti p_1\leq \ti c_2$; a contradiction.

Therefore, $c_2>\ti c_2$. In particular, $h\geq 2$. Multiplying \eqref{eq:3_case03} by $c_2$ and subtracting \eqref{eq:3_case03} we get $$\gamma (c_2-1)\geq a(\beta c_2+\ti c_1)+x(\ti c_1-c_2).$$ Suppose $\ti c_1\geq c_2$. Then $\gamma-1\geq a(\beta+1)$, so $\gamma=5$, $a=2$ and $\beta=1$. Then $\ti c_1=c_1-a=kc_2-2$, so the inequality gives $5c_2-5\geq 2(c_2+kc_2-2)$, hence $3>2k$; a contradiction.

Thus $c_2>\ti c_1$.  We rewrite again \eqref{eq:1_case03} and \eqref{eq:2_case03}:

\begin{align}
  \gamma+(k+\beta )c_2 &=(\ti j-2)\ti c_1+\ti p_1+P+\ti P.\label{eq:7n_case03}\\
  \gamma +\beta kc_2^2+2kc_2\ti c_1 &=(\ti j-1)\ti c_1^2+\ti c_1\ti p_1+\sum\limits_{i\geq 2}p_ic_i+\sum\limits_{i\geq 2}\ti p_i\ti c_i.\label{eq:8n_case03}
\end{align}
Multiplying \eqref{eq:7n_case03} by $c_2$ and subtracting \eqref{eq:8n_case03} we get
\begin{equation*}\gamma(c_2-1)\geq ((\beta -1)(k-1)-1)c_2^2+(c_2-\ti c_1)((\ti j-2)\ti c_1+\ti p_1)+\ti c_1(c_1+a).\end{equation*}
It follows that $\beta=1$. Indeed, if $\beta\geq 2$ then the above inequality gives $\gamma c_2>\ti c_1 (c_1+a)=\ti c_1(2c_1-\ti c_1)\geq 2(2c_1-\ti c_1)>2(2c_1-c_2)\geq 6c_2$, which is impossible, because $\gamma\leq 5$. Note that $c_2\geq \ti c_1+1\geq 3$, so the inequality implies also that $c_2\geq 4$. Indeed, for $c_2=3$ we get $\ti c_1=2$, so $2\gamma+9\geq \ti c_1(2c_1-\ti c_1)\geq 2(4c_2-2)=20$, which is impossible. Put $f=c_2-\ti c_1$. We have $f\geq 1$ and $a=(k-1)c_2+f$, so we can rewrite the above inequality as
\begin{equation}c_2(\gamma+c_2)-\gamma\geq f((\ti j-2)\ti c_1+\ti p_1)+\ti c_1((2k-1)c_2+f)\label{eq:9n_case03}.\end{equation}
It follows that \begin{equation}\gamma+c_2>f(\ti j-2)\frac{\ti c_1}{c_2}+\ti c_1(2k-1).\label{eq:9f_case03}\end{equation}

Consider the case $\ti j\geq 6$. Suppose $\ti c_1\geq \frac{1}{2}c_2$. Then \eqref{eq:9f_case03} gives $2\gamma-1\geq f(\ti j-2)+c_2(2k-3)$, so since $c_2\geq 4$, we get $\gamma=5$, $k=2$ and $f=1$. In particular, $\ti c_1=c_2-1$. The inequality \eqref{eq:9n_case03} gives $c_2(5+c_2)-5\geq 4(c_2-1)+1+(c_2-1)(3c_2+1)$, hence $c_2(3-2c_2)\geq 1$; a contradiction. Thus $c_2>2\ti c_1$. We have $(h-1)+(\ti h-1)=\gamma+\varepsilon-\ti j\leq 7-\ti j\leq 1$ and $2kc_2\ti c_1\geq 4c_2\ti c_1\geq 8\ti c_1^2$. It follows from \eqref{eq:8n_case03} that $\gamma+2c_2^2+8\ti c_1^2\leq \ti j\ti c_1^2+(h-1)c_2^2+(\ti h-1)\ti c_2^2\leq \ti j\ti c_1^2+c_2^2+\ti c_1^2$, hence $c_2^2<(\ti j-7)\ti c_1^2$, so $\ti j\geq 8$. But then $h+\ti h\leq 1$; a contradiction.

Thus $\ti j\leq 5$. Suppose $\ti c_1\geq \frac{2}{3} c_2$. Then \eqref{eq:9f_case03} gives $\gamma+c_2>\frac{2}{3}f(\ti j-2)+\frac{2}{3}(2k-1)c_2,$ so $14\geq 3\gamma-1\geq 2f(\ti j-2)+(4k-5)c_2$. Since $c_2\geq 4$, we get $k=2$, $f=1$, $\ti j=3$, $\gamma=5$ and $c_2=4$. But then $\ti c_1=3$ and \eqref{eq:9f_case03} fails; a contradiction.

Thus $c_2>\frac{3}{2}\ti c_1$. Since $(h-1)+(\ti h-1)=\gamma+\varepsilon-\ti j\leq 4$, we have $\sum\limits_{i\geq 2}\ti p_i\ti c_i\leq 4\ti c_2^2\leq \ti c_1^2$. Hence $(\ti j-1)\ti c_1^2+\ti c_1\ti p_1+\sum\limits_{i\geq 2}\ti p_i\ti c_i<6\ti c_1^2\leq 4c_2\ti c_1\leq 2kc_2\ti c_1$. Then \eqref{eq:8n_case03} gives
\begin{equation}kc_2^2<\sum\limits_{i\geq 2}p_ic_i\leq (h-1)c_2^2\leq 4c_2^2,\label{eq:10n_case03}\end{equation}
so $k\leq 3$ and $h-1\geq k+1\geq 3$. If $k=3$ then $h-1=4$, $\ti h=1$, $\ti j=3$ and $\varepsilon +\gamma=7$, so $\gamma=5$, $\varepsilon=2$ and $t=2$ by \ref{prop:basic_inequality}. But in the latter case from \eqref{eq:8n_case03} we get $3c_2^2+6c_2\ti c_1<2\ti c_1^2+\ti c_1\ti p_1+3c_2^2+c_2<3\ti c_1^2+3c_2^2+c_2$, which is impossible, because $c_2>\ti c_1$. Hence $k=2$, i.e.\ $c_1=2c_2$. We claim that $c_2=p_2$ and $c_3=p_3$.  Indeed, if $c_2>p_2$ then $\sum\limits_{i\geq 2}p_ic_i\leq (c_2-c_3)c_2+3c_3^2\leq c_2^2-c_3\cdot 2c_3+3c_3^2=c_2^2+c_3^2\leq \frac{5}{4}c_2^2$ and if $c_3>p_3$ then $\sum\limits_{i\geq 2}p_ic_i=p_2c_2+p_3c_3+p_4c_4\leq c_2^2+(c_2-c_4)c_2+c_4^2< 2c_2^2$, and both inequalities contradict \eqref{eq:10n_case03}. We obtain $h-1\geq 3$.

Suppose that $h-1=3$ and $\ti h-1=0$. Then $Q_0$ is a fork. Because $j=0$ and $\ti j\geq 3$, we have $L_\8'\leq -2$ and $Q_0$ is not of quotient type. $Q_0$ has only one $(-2)$-twig (produced by $\binom{p_1}{c_1}$) and since $k=2$, the twig is irreducible. Note that $\gamma=\ti j+3-\varepsilon\geq \ti j\geq 3$. By \ref{lem:Y_geometry} $Y$ is not ruled and the inequality $\frac{1}{c_4}+\frac{1}{\ti c_1}+\frac{1}{\gamma}\geq 1$ holds. Since $c_4=c_2\geq 4$, we get $\ti c_1=2$. Then $\ti p_1=1$ and \eqref{eq:7n_case03} gives $\gamma=3+p_4-c_2<3$; a contradiction.

Hence $(h-1)+(\ti h-1)=4$. Then $\gamma=5$, $\varepsilon=2$, $\ti j=3$ and $t=2$. Also, $Q_0$ is not of quotient type, so by \ref{lem:k(Y)} $\kappa(Y)\geq 0$. Then \ref{lem:Y_geometry} gives $\frac{1}{c_h}+\frac{1}{\ti c_{\ti h}}\geq \frac{4}{5}$, hence $\{c_h,\ti c_{\ti h}\}=\{2,3\}$ or $\{2,2\}$. Because $c_4=c_2\geq 4$, we have $h\geq 5$, so $h=5$ and $\ti h=1$. From \eqref{eq:7n_case03} we get $3+c_2=\ti c_1+p_4$. Since $\ti c_1\leq 3$ and $p_4\leq c_2$, we get $\ti c_1=3$ and $p_4=c_2$. Hence $p_4=c_4$, so $c_2=c_5\leq 3$; a contradiction, because $c_2\geq 4$.

\bca $x<0$. \eca

We have $2c_1\geq \ti c_1+\ti p_1+p_1$. Since $x=2c_1-p_1-\ti p_1-\ti c_1-\ti c_1(\ti j-3)$, $x<0$ gives $\ti j\geq 4$. We have $$(h-1)+(\ti h-1)=\gamma+\varepsilon-\ti j\leq 7-\ti j\leq 3.$$ We rewrite the formulas \eqref{eq:3_case03}, \eqref{eq:4_case03}:
\begin{align}
  \gamma+x &=P+\ti P,\label{eq:5_case03}\\
  \gamma+\beta a c_2 +a\ti c_1 &=-x\ti c_1 +\sum\limits_{i\geq 2}p_ic_i+\sum\limits_{i\geq 2}\ti p_i\ti c_i.\label{eq:6_case03}
\end{align}
Note that, because $x<0$, \eqref{eq:5_case03} gives $P+\ti P\leq \gamma-1\leq 4$.

Consider the case $(h-1)+(\ti h-1)=3$. Then $P+\ti P\geq 4$ hence, by \eqref{eq:5_case03}, $\gamma=5$, $x=-1$  and $P+\ti P=4$. Also  $\varepsilon=2$, $t=2$ and $\ti j=4$. In particular, $p_h=\ti p_h=1$. Since $P+\ti P=4$, it cannot happen that $h-1=3$ or $\ti h-1=3$. If $h-1=2$ and $\ti h-1=1$ then $p_2=c_3=2$ and \eqref{eq:6_case03} gives $5+2c_2+2\ti c_1\leq 5+\beta a c_2+a \ti c_1=\ti c_1+2c_2+2+\ti c_2$, so $3+\ti c_1\leq \ti c_2$; which is impossible. Similarly, if $h-1=1$ and $\ti h-1=2$ then $\ti p_2=\ti c_3=2$ and then \eqref{eq:6_case03} gives $5+2c_2+2\ti c_1\leq 5+\beta a c_2+a\ti c_1=\ti c_1+c_2+2\ti c_2+2$, so $3+c_2+\ti c_1\leq 2\ti c_2\leq \ti c_1$; again a contradiction.

Consider the case $(h-1,\ti h-1)=(2,0)$. Then $\gamma+\varepsilon=\ti j+2$ and $\ti P=0$. We have $P=p_2+p_3\geq 3$.  Suppose that $p_2=2$. Then $p_3=1$ and $c_3=2$, so $P=3$ and \eqref{eq:5_case03} gives $x=3-\gamma$. From \eqref{eq:6_case03} we get $\gamma+\beta a c_2+a \ti c_1= (\gamma-3)\ti c_1+2c_2+2\leq 2\ti c_1+2c_2+2$, this is a contradiction. Suppose that $p_2=3$. Then $P=4$, $p_3=1$, $c_3=3$, $x=-1$, $\gamma=5$ and \eqref{eq:6_case03} gives $5+\beta a c_2+a \ti c_1=\ti c_1+3c_2+3.$ It follows that $\beta =1$, $a=2$ and $c_2=\ti c_1+2$. We get $c_1=\ti c_1+2=c_2$; a contradiction.

Consider the case $(h-1,\ti h-1)=(0,2)$. Now $\gamma+\varepsilon=\ti j+2$ and $P=0$. As above we have $\ti p_3=1$ and $\ti p_2=2$ or $3$. Suppose that $\ti p_2=3$. Then $\ti P=4$, $\ti p_3=1$, $\ti c_3=3$, $\gamma=5$ and $x=-1$. From \eqref{eq:6_case03} we get $5+\beta a +a\ti c_1=\ti c_1+3\ti c_2+3$. From that we obtain $a=2$, $\ti k=2$ and $4+2\beta=\ti c_2$. Then $\ti c_2$ is even and divisible by $\ti c_3=3$, hence $\ti p_1\geq \ti c_2\geq 6$. But then $x=4-(\ti j-4)\ti c_1-p_1-\ti p_1\leq -2$; a contradiction. Thus $\ti p_2=2$. Then $\ti p_3=1$, $\ti c_3=2$, $\ti P=3$ and $x=3-\gamma$, so \eqref{eq:6_case03} gives $\gamma+\beta a+a\ti c_1=(\gamma-3)\ti c_1+2\ti c_2+2\leq (\gamma-2)\ti c_1^2+2$. It follows that $a=2$ and $\gamma=5$. But then the equation implies that $\gamma$ is even; a contradiction.

Consider the case $(h-1,\ti h-1)=(1,1)$. We still have $\gamma+\varepsilon=\ti j+2$. Suppose that $p_2+\ti p_2=4$. Then $\gamma=5$, $x=-1$ and $p_2, \ti p_2\leq 3$. If $a\geq 3$ then $\beta a c_2\geq p_2c_2$ and $a \ti c_1\geq \ti c_1+\ti p_2\ti c_2$, which is inconsistent with \eqref{eq:6_case03}. So $a=2$ and we have $-1=x=4-(\ti j-4)\ti c_1-p_1-\ti p_1$, i.e.\ $5=(\ti j-4)\ti c_1+p_1+\ti p_1$. We obtain $5\geq p_1+\ti p_1\geq c_2+\ti c_2\geq p_2+1+\ti p_2+1=6$; a contradiction. Suppose that $p_2+\ti p_2=3$. If $x=-1$ then \eqref{eq:6_case03} gives $\gamma +\beta a c_2+a \ti c_1=\ti c_1+p_2c_2+\ti p_2\ti c_2\leq \ti c_1+2c_2+2\ti c_2$, which is impossible. So $x=-2$ and $\gamma=5$. We have $5+\beta a c_2+a \ti c_1=2\ti c_1+p_2c_2+\ti p_2\ti c_2$, so again $a=2$. Then the equality $x=-2$ gives $6-(\ti j-4)\ti c_1=p_1+\ti p_1\geq c_2+\ti c_2\geq p_2+\ti p_2+2=5$, so $\ti j=4$, $p_1=c_2$, $\ti p_1=\ti c_2$ and $c_2+\ti c_2=6$. Now \eqref{eq:6_case03} gives $5+(2\beta +\ti p_2-p_2)c_2=6\ti p_2$. We have $(\ti p_2,p_2)\neq (1,2)$, because otherwise $5+(2\beta -1)c_2=6$, which is impossible, because $c_2>p_2=2$. So $\ti p_2=2$ and $p_2=1$ and then $(2\beta +1)c_2=7$, which is again a contradiction, because $c_2>1$.

Thus we may assume that $p_2=\ti p_2=1$. In particular, $c_2, \ti c_2\geq 2$. We have $x=2a-(\ti j-4)\ti c_1-p_1-\ti p_1$, so $2a-x\geq p_1+\ti p_1$. From \eqref{eq:6_case03} we get $\gamma+(a+x)\ti c_1+(\beta a -1)c_2=\ti c_2$, hence $a+x\leq 0$. Suppose that $a+x= 0$. Then $\ti c_2= \gamma+c_2(\beta a-1)$. We obtain that $2a-x\geq p_1+\ti p_1\geq c_2+\ti c_2= c_2+\gamma+c_2(\beta a -1)$, so $\beta a c_2\leq 2a-x-\gamma=2a-2<2a$, which gives $c_2<2$; a contradiction. We infer that $2-\gamma=x\leq -a-1\leq -3$, so $x=-3$, $\gamma=5$ and $a=2$. From \eqref{eq:6_case03} we get $5+2\beta c_2=\ti c_1+c_2+\ti c_2$. Because $x=-3$, we have $7=(\ti j-4)\ti c_1+p_1+\ti p_1$. If $\ti j>4$ then $7\geq \ti c_1+p_1+\ti p_1\geq \ti c_1+c_2+\ti c_2=5+2\beta c_2\geq 9$, which is impossible. Therefore, $\ti j=4$ and $p_1+\ti p_1=7$. From \eqref{eq:6_case03} we get $5+2\beta c_2=\ti c_1+c_2+\ti c_2=c_1-2+c_2+\ti c_2$ i.e.\ $7-(c_2+\ti c_2)=(k-2\beta)c_2$. Clearly, if $c_2$ is even then $\ti c_2$ is odd. In particular, $c_2+\ti c_2\geq 5$. Because $c_2\geq 2$ divides $7-(c_2+\ti c_2)$, we get $c_2+\ti c_2=7$ and $k=2\beta$. But $\gcd(k,\beta)=1$, so $\beta=1$ and $k=2$. We have $7-c_2=\ti c_2=\ti p_1<\ti c_1=c_1-2=2c_2-2$, so $c_2>3$. By \eqref{eq:noasy_case03} $c_2=p_1\leq 2a=4$, so $c_2=4$ and hence $\ti c_2=3$. By \ref{lem:k(Y)} $\kappa(Y)\geq 0$, so \ref{lem:Y_geometry} gives $\frac{1}{3}+\frac{1}{4}+\frac{1}{5} \geq 1-\frac{1}{|\Gamma(Q_0)|}$. We obtain $d(Q_0)\leq |\Gamma(Q_0)|\leq 4$. Since $L_\8'$ is not a $(-1)$-curve, we check easily that $d(Q_0)>4$; a contradiction.

Consider the case $(h-1,\ti h-1)=(1,0)$. We have now $c_2\geq 2$. Then \eqref{eq:5_case03} gives $x=p_2-\gamma\geq -4$ and \eqref{eq:6_case03} takes the form
 \begin{equation} \gamma+\beta a c_2+a\ti c_1 =-x\ti c_1+p_2c_2.\label{eq:6p_case03} \end{equation}
From this we get
\begin{equation}\gamma+(\beta a-\gamma-x)c_2+(a +x)\ti c_1=0.\label{eq:7_case03}\end{equation} which, because $\ti c_1=c_1-a=kc_2-a$, gives
\begin{equation}c_2(\gamma+x-k(a+x)-\beta a)=\gamma-a(a+x).\label{eq:7m_case03}\end{equation}
In particular, by \eqref{eq:7_case03} we have $a\beta <\gamma+x$ or $x<-a$. Hence $a\beta<5+x\leq 4$ or $a<-x\leq 4$, so $a\leq 3$. Suppose $a=3$. If $3\beta<\gamma+x$ then $\beta=1$, $\gamma=5$ and $x=-1$, so $\gamma-a(a+x)=-1$ and $c_2|1$ by \eqref{eq:7m_case03}, which is impossible. Thus $x<-a$, i.e.\ $x=-4$. Then $\gamma=5$ and \eqref{eq:7m_case03} gives $c_2(k+1-3\beta)=8$, hence $p_1=(k-\beta)c_2=8+(2\beta-1)c_2\geq 10$. But by the definition of $x$ and $a$ we have  $p_1<p_1+\ti p_1+(\ti j-4)\ti c_1=2a-x=10$; a contradiction.

\noin Therefore, $a=2$. For $x=-1$ \eqref{eq:7m_case03} gives $c_2(\gamma-1-k-2\beta)=\gamma$, so $\gamma>1+k+2\beta\geq 5$, which is impossible. For $x=-2$ it gives $c_2(\gamma-2-2\beta)=\gamma$, so $\gamma\geq 3+2\beta$ and consequently $\gamma=5$, $\beta=1$ and $c_2=5$. But in the latter case $p_1\geq c_2>4=2a$, which contradicts \eqref{eq:noasy_case03}. Suppose $x=-3$. Then \eqref{eq:7m_case03} gives $c_2(\gamma-3+k-2\beta)=\gamma+2$. If $\gamma=5$ then $c_2|7$, hence $p_1\geq c_2\geq 7$. But by the definition of $x$ we have $p_1<p_1+\ti p_1+(\ti j-4)\ti c_1=2a-x=7$; a contradiction. Thus $x=-4$. Then $\gamma=5$, $p_2=1$ and $c_2(2(k-\beta)+1)=9$. Since $k-\beta\geq 1$, it follows that $c_2=3$ and $k=\beta+1$. In particular, $p_1=c_2=3$. We have $\ti c_1=c_1-2=3k-2$. By the definition of $x$ we have $\ti p_1+(\ti j-4)(3k-2)=5$, hence $\ti j\leq 5$.  By \ref{lem:k(Y)} $\kappa(Y)\geq 0$, so by \ref{lem:Y_geometry} $\frac{1}{3}+\frac{1}{3k-2}+\frac{1}{5}\geq 1-\frac{1}{|\Gamma(Q_0)|}$, hence $|d(Q_0)|\leq |\Gamma(Q_0)|\leq 4$. We check easily that this is not the case (note $L_\8'$ is a $(-3)$-curve); a contradiction.

Consider the case $(h-1,\ti h-1)=(0,1)$.

\noin We have $\gamma+x=\ti p_2$ and we can write the equation \eqref{eq:6_case03} in the form
\begin{equation} \gamma+\beta a+(a+x)\ti c_1 =\ti p_2\ti c_2.\label{eq:10m_case03} \end{equation}
Because $\beta=c_1-p_1=\ti c_1+a-p_1$, we can rewrite it as
\begin{equation} (2a+x)\ti c_1=\ti p_2 \ti c_2+a(p_1-a)-\gamma.\label{eq:11m_case03} \end{equation}
By the definition of $x$ we have
\begin{equation}p_1+\ti p_1+(\ti j-4)\ti c_1=2a-x.\label{eq:12m_case03} \end{equation}
We have $\ti p_2\ti  c_2=(\gamma+x)\ti c_2\leq \frac{1}{2}(\gamma+x)\ti c_1$, so \eqref{eq:10m_case03} gives $a+x<\frac{1}{2}(\gamma+x)$, hence $a\leq \frac{1}{2}(\gamma-1-x)\leq \frac{1}{2}(\gamma+3)\leq 4$. Suppose $a=4$. Then $\gamma=5$, $x=-4$ and $\ti p_2=1$, so \eqref{eq:10m_case03} gives $5+4\beta=\ti c_2$. But \eqref{eq:12m_case03} gives $\ti c_2\leq \ti p_1<2a-x=12$, so $\beta=1$ and $\ti c_2=9$. By \ref{lem:k(Y)} $\kappa(Y)\geq 0$. We have now $c_1=\ti c_1+4=9\ti k+4\geq 22$, so \ref{lem:BMY} gives $\frac{1}{22}+\frac{1}{9}+\frac{1}{5}\geq \frac{1}{|\Gamma(Q_0)|}$, i.e.\ $|\Gamma(Q_0)|\leq 2$; a contradiction. Thus $a\leq 3$.

Note that by \eqref{eq:12m_case03} $p_1+\ti p_1\leq 2a-x\leq 10$, so $\ti p_2+1\leq \ti c_2\leq \ti p_1\leq 9$ and $p_1\leq 9$. It follows that the expression $M=\ti p_2\ti c_2+a(p_1-a)-\gamma$ is bounded. Suppose $2a+x\neq 0$. Then for each possible choice of $a, x, \gamma, p_1, \ti p_1, \ti c_2, \ti p_2$ and $\gamma$ we can compute $M$, then $\ti c_1=M/(2a+x)$ and then $c_1=\ti c_1+a$. A straightforward verification shows that there are two solutions, both with $\ti j=4$: $(c_1,p_1;\ti c_1,\ti p_1,\ti c_2,\ti p_2)=(6,5;4,2,2,1)$ and $(9,8;6,2,2,1)$. For both \eqref{eq:noasy_case03} fails; a contradiction. Thus $2a+x=0$, which gives $a=2$, $x=-4$ and hence $\gamma=5$ and $\ti p_2=1$. In particular, $Q_0$ is a chain. The equation \eqref{eq:11m_case03} gives $\ti c_2+2p_1=9$. It follows that $\ti c_2$ is odd, so $\ti c_2\geq 3$. By \ref{lem:k(Y)} $\kappa(Y)\geq 0$, so by \ref{lem:Y_geometry} $\frac{1}{c_1}+\frac{1}{\ti c_2}\geq \frac{4}{5}-\frac{1}{|\Gamma(Q_0)|}$. Because $j=0$, we have $(L_\8')^2\leq -3$, so $d(Q_0)\geq 3$. But $c_1=\ti c_1+a\geq 2\ti c_2+2\geq 8$ and we check that the above inequality fails.

We are left with the case $(h-1,\ti h-1)=(0,0)$. We have $x=-\gamma$. The equation \eqref{eq:6_case03} gives $\gamma+\beta a=(\gamma-a)\ti c_1$, hence $a<\gamma$, so $a\leq 4$ and $$(\gamma-2a)\ti c_1=\gamma+a(a-p_1).$$ The definition of $x$ gives $$p_1+\ti p_1+(\ti j-4)\ti c_1=2a+\gamma,$$ so $p_1,\ti p_1\leq 2a+\gamma-1\leq 12$. Assume $\gamma\neq 2a$. For every $a\in\{2,3,4\}$, $\gamma\in\{2,3,4,5\}$ and every $p_1,\ti p_1\leq 2a-\gamma-1$ we computed $\ti c_1=(\gamma+a(a-p_1))/(\gamma-2a)$, $c_1=\ti c_1+a$ and $\ti j=(2a-\gamma-p_1-\ti p_1)/\ti c_1$ and we checked that only three solutions satisfy \eqref{eq:noasy_case03}. These are $(c_1,p_1,\ti c_1,\ti p_1)=(9,1,7,1)$, $(7,2,5,2)$ and $(7,6,4,1)$, all with $\gamma=\ti j=5$. In all cases $\gamma+t\geq 6$, so $\kappa(Y)\geq 0$. Since $d(Q_0)>2$ we check easily that the log BMY inequality (\ref{lem:BMY}) fails; a contradiction. Assume $\gamma=2a$. Then $a=2$ and $\gamma=4$, so the equation gives $p_1=4$. Because $Q_0$ is a chain, we have $\kappa(Y)\geq 0$, so $\frac{1}{\ti c_1+2}+\frac{1}{\ti c_1}\geq 1-\frac{1}{d(Q_0)}$. By the definition of $x$ we get $\ti p_1+(\ti j-4)\ti c_1=4$. If $\ti j>4$ then $\ti j=5$, $\ti p_1=1$ and $\ti c_1=3$, so $d(Q_0)\leq 4$; a contradiction. Thus $\ti j=4$ and $\ti p_1=4$. Then $\ti c_1\geq 5$, so $d(Q_0)\leq 2$; a contradiction.
\end{proof}

\bcor $\ti j\leq 6$. \ecor

\begin{proof} By \ref{lem:Psi_properties}(ii) $\gamma'=\gamma$, so the equation \eqref{eq:h+hB} gives $\ti j+1=\gamma+\varepsilon+2-h-\ti h$. By \ref{prop:no_simple_branch} $h,\ti h\geq 1$, so we get $\ti j+1\leq \gamma+\varepsilon$. Suppose $\ti j\geq 7$. Then $\gamma+\varepsilon\geq 8$, so \ref{prop:gamma<=5} gives $\gamma=5$ and $\varepsilon\geq 3$. We get a contradiction with \ref{prop:basic_inequality}.
\end{proof}

Proposition \ref{prop:j>0} implies that for $\C^*$-embeddings $U\mono \C^2$ which do not admit a good asymptote one can choose coordinates in which the type of $U$ at infinity is $(j,\ti j)$ for some $j,\ti j>0$. But $j>0$ if and only if the line tangent to $\lambda$, which in the spirit of elementary planar geometry, should be called \emph{an asymptote} of $U$, is different than $L_\8$. An analogous remark holds for $\ti j$. In view of results in \cite{CKR-Cstar_good_asymptote} we obtain the following result, which shows that most closed $\C^*$-embeddings are hyperbola-like in suitable coordinates.

\begin{proposition}\label{asymptotes} Let $U\subset \C^2$ be a closed $\C^*$-embedding. Then we can choose coordinates on $\C^2$ with respect to which $U$ has at least one asymptote (in the sense of elementary planar geometry). If the embedding is not as in cases 6.8.1.2(b) and 6.8.1.3 of \cite{CKR-Cstar_good_asymptote}, in particular if it does not admit a good asymptote, then we can choose coordinates with respect to which $U$ has two distinct asymptotes. \end{proposition}

\section{\texorpdfstring{Type $(1,1)$}{Type (1,1)}} We keep the notation from previous sections. In particular, $(j,\ti j)$ is the type at infinity of the $\C^*$-embedding $U\mono S=\C^2$. By \ref{prop:j>0} and \ref{lem:j<=1} we have $j=1$. Here we show that $\ti j\geq 2$.

Assume $U$ is of type $(j,\ti j)=(1,1)$. The formulas \eqref{eq:1} and \eqref{eq:2} read as
\begin{align*}\gamma+c_1+\ti c_1 &=p_1+\ti p_1+P+\ti P.\\
\gamma+2c_1\ti c_1 &=c_1p_1+\ti c_1\ti p_1+\sum\limits_{i\geq 2}p_ic_i+\sum\limits_{i\geq 2}\ti p_i\ti c_i.
\end{align*}
By \ref{lem:c1B-p1>=2} (and its analogue for $\ti \lambda$) we have
\begin{equation}\ti c_1-p_1\geq 2 \text{\ \ and\ \ } c_1-\ti p_1\geq 2\label{eq:noasy_case11}.\end{equation}
We may assume that $c_1\geq \ti c_1$. Let $x=c_1+\ti c_1-p_1-\ti p_1$. We rewrite the formulas in the following form. \begin{align}
\gamma+x &=P+\ti P.\label{eq:1_case11}\\
\gamma+c_1x&=(c_1-\ti c_1)(c_1-\ti p_1) +\sum\limits_{i\geq 2}p_ic_i+\sum\limits_{i\geq 2}\ti p_i\ti c_i.\label{eq:2_case11}
\end{align}
Multiplying \eqref{eq:1_case11} by $c_1$ and subtracting \eqref{eq:2_case11} we obtain
\begin{equation}
  \gamma(c_1-1)=(c_1-\ti c_1)(c_1-\ti p_1)+\sum\limits_{i\geq 2}p_i(c_1-c_i)+\sum\limits_{i\geq 2}\ti p_i(c_1-\ti c_i).\label{eq:3_case11}
\end{equation}
Since $c_1\geq \ti c_1\geq 2\ti c_2$, we have $c_1-\ti c_i\geq \frac{c_1}{2}$ for $i\geq 2$. Also $c_1-c_i\geq \frac{c_1}{2}$ for $i\geq 2$. We get $$\gamma(c_1-1)\geq (c_1-\ti c_1)(c_1-\ti p_1)+\frac{c_1}{2}(P+\ti P),$$ hence $P+\ti P\leq 2\gamma-1$. Then \eqref{eq:1_case11} gives $x\leq \gamma-1\leq 4$. But we have $x=c_1+\ti c_1-p_1-\ti p_1=c_1-\ti p_1+\ti c_1-p_1\geq 2+2=4$ by \eqref{eq:noasy_case11}, so $x=4$, $\gamma=5$ and $P+\ti P=9$. Also $c_1-\ti p_1=2$ and $\ti c_1-p_1=2$. We have $(h-1)+(\ti h-1)=3+\varepsilon$ by \eqref{eq:h+hB}. Let $c_1-p_1=\beta c_2$, $\ti c_1-\ti p_1=\ti \beta \ti c_2$. We have
\begin{equation}x=c_1-p_1+\ti c_1-\ti p_1=\beta c_2+\ti \beta \ti c_2=4.\label{eq:(i)_case11}
\end{equation}
By \eqref{eq:2_case11}
\begin{equation}5+2c_1+2\ti c_1=\sum\limits_{i\geq 2}p_ic_i+\sum\limits_{i\geq 2} \ti p_i\ti c_i.\label{eq:(ii)_case11}
\end{equation}

If $h>1$ and $\ti h>1$ then from \eqref{eq:(i)_case11} we get $c_2=\ti c_2=2$ and $\beta =\ti \beta =1$, hence $P+\ti P=2(h-2)+1+2(\ti h-2)+1=2h+2\ti h-2$ is even. But we have already shown that $P+\ti P=9$, so $h=1$ or $\ti h=1$. Suppose that $h=1$. Then $\ti h-1=3+\varepsilon\geq 3.$ We have also $\ti P=9$. \eqref{eq:(i)_case11} gives $\beta +\ti \beta \ti c_2=4$. If $\beta =2$ then $\ti c_2=2$, which implies that $\sum\limits_{i\geq 2} \ti p_i\ti c_i$ is even; a contradiction with \eqref{eq:(ii)_case11}. Thus $\beta =1$ and then $\ti \beta =1$, $\ti c_2=3$. We get $9=\ti P=3(\ti h-2)+\ti p_h$. It implies that $3$ divides $p_h$. This is impossible, because $p_{\ti h}<\ti c_h=3$. Thus $\ti h=1$. But then we reach a contradiction the same way as for $h=1$.

\bibliographystyle{amsalpha}
\bibliography{bibl}

\end{document}